\documentclass[letterpaper]{article} 
\usepackage{aaai25}  
\usepackage{times}  
\usepackage{helvet}  
\usepackage{courier}  
\usepackage[hyphens]{url}  
\usepackage{graphicx} 
\usepackage{natbib}  
\usepackage{caption} 
\usepackage{algorithm}
\usepackage{algorithmic}

\usepackage[utf8]{inputenc}
\usepackage{amsmath,amssymb,amsfonts,mathrsfs}
\usepackage{amsthm}
\usepackage{booktabs}
\usepackage{bm}
\usepackage{subfig}
\usepackage{multirow}
\usepackage{diagbox}

\usepackage{enumitem}
\setlist{nolistsep}

\allowdisplaybreaks[4]

\newtheorem{theorem}{Theorem}
\newtheorem{lemma}[theorem]{Lemma}

\newtheorem{corollary}[theorem]{Corollary}
\newtheorem{definition}[theorem]{Definition}

\DeclareMathOperator*{\argmin}{arg\,min}

\usepackage{newfloat}
\usepackage{listings}
\DeclareCaptionStyle{ruled}{labelfont=normalfont,labelsep=colon,strut=off} 
\floatstyle{ruled}
\newfloat{listing}{tb}{lst}{}
\floatname{listing}{Listing}
\title{De-singularity Subgradient for the $q$-th-Powered $\ell_p$-Norm Weber Location Problem}
\author {
    Zhao-Rong Lai\textsuperscript{\rm 1,\rm 2},
    Xiaotian Wu\textsuperscript{\rm 1},
    Liangda Fang\textsuperscript{\rm 3,\rm 4},
    Ziliang Chen\textsuperscript{\rm 5},
    Cheng Li\textsuperscript{\rm 2}\thanks{Corresponding author.}
}
\affiliations {
    \textsuperscript{\rm 1}College of Cyber Security, Jinan University\\
    \textsuperscript{\rm 2}Department of Mathematics, College of Information Science and Technology, Jinan University\\
    \textsuperscript{\rm 3}Department of Computer Science, College of Information Science and Technology, Jinan University\\
    \textsuperscript{\rm 4}Pazhou Laboratory\\
    \textsuperscript{\rm 5}Research Institute of Multiple Agents and Embodied Intelligence, Peng Cheng Laboratory\\
    \{laizhr, wxiaotian, fangld\}@jnu.edu.cn, c.ziliang@yahoo.com, licheng@stu2020.jnu.edu.cn
}

\usepackage{bibentry}

\makeatletter
\newenvironment{breakablealgorithm}
{
	\begin{center}
		\refstepcounter{algorithm}
		\hrule height.8pt depth0pt \kern2pt
		\renewcommand{\caption}[2][\relax]{
			{\raggedright\textbf{\ALG@name~\thealgorithm} ##2\par}%
			\ifx\relax##1\relax 
			\addcontentsline{loa}{algorithm}{\protect\numberline{\thealgorithm}##2}%
			\else 
			\addcontentsline{loa}{algorithm}{\protect\numberline{\thealgorithm}##1}%
			\fi
			\kern2pt\hrule\kern2pt
		}
	}{
		\kern2pt\hrule\relax
	\end{center}
}
\makeatother

\begin{document}

\maketitle

\def \bbI {\mathbb I}
\def \bbC {\mathbb C}
\def \bbR {\mathbb R}
\def \bbN {\mathbb N}
\def \bbS {\mathbb S}
\def \bbZ {\mathbb Z}
\def \bbRn {\bR^n}
\def \bbRm {\bR^m}
\def \bbRmn {\bR^{m\times n}}
\def \bbRnm {\bR^{n\times m}}
\def \bbRnn {\bR^{n\times n}}
\def \bbRmm {\bR^{m\times m}}

\def \tr {\mathrm{tr}}
\def \var {\mathrm{var}}
\def \cov {\mathrm{cov}}
\def \sgn {\mathrm{sgn}}
\def \range {\mathrm{range}}
\def \prox {\mathrm{prox}}
\def \diag {\mathrm{diag}}

\def \bA {\bm A}
\def \bB {\bm B}
\def \bC {\bm C}
\def \bD {\bm D}
\def \bH {\bm H}
\def \bI {\bm I}
\def \bK {\bm K}
\def \bL {\bm L}
\def \bP {\bm P}
\def \bQ {\bm Q}
\def \bR {\bm R}
\def \bS {\bm S}
\def \bT {\bm T}
\def \bU {\bm U}
\def \bV {\bm V}
\def \bW {\bm W}
\def \bX {\bm X}
\def \bY {\bm Y}
\def \bZ {\bm Z}

\def \ba {\bm a}
\def \bb {\bm b}
\def \bc {\bm c}
\def \bd {\bm d}
\def \be {\bm e}
\def \bh {\bm h}
\def \bl {\bm l}
\def \bp {\bm p}
\def \bq {\bm q}
\def \br {\bm r}
\def \bs {\bm s}
\def \bu {\bm u}
\def \bv {\bm v}
\def \bw {\bm w}
\def \bx {\bm x}
\def \by {\bm y}
\def \bz {\bm z}

\def \mD {\mathcal{D}}
\def \mH {\mathcal{H}}
\def \mI {\mathcal{I}}
\def \mJ {\mathcal{J}}
\def \mK {\mathcal{K}}
\def \mL {\mathcal{L}}
\def \mN {\mathcal{N}}

\def \sD {\mathscr{D}}

\def \fk {\mathfrak{k}}

\def \FixT {\text{Fix}{(T)}}
\def \sign {\mathrm{sign}}
\def \mO {\mathcal{O}}
\def \btheta {\bm{\theta}}

\newcommand{\ud}{\,\mathrm{d}}	
\newcommand\leqs{\leqslant}
\newcommand\geqs{\geqslant}

\begin{abstract}
The Weber location problem is widely used in several artificial intelligence scenarios. However, the gradient of the objective does not exist at a considerable set of singular points. Recently, a de-singularity subgradient method has been proposed to fix this problem, but it can only handle the $q$-th-powered $\ell_2$-norm case ($1\leqslant q<2$), which has only finite singular points. In this paper, we further establish the de-singularity subgradient for the $q$-th-powered $\ell_p$-norm case with $1\leqslant q\leqslant p$ and $1\leqslant p<2$, which includes all the rest unsolved situations in this problem. This is a challenging task because the singular set is a continuum. The geometry of the objective function is also complicated so that the characterizations of the subgradients, minimum and descent direction are very difficult. We develop a $q$-th-powered $\ell_p$-norm Weiszfeld Algorithm without Singularity ($q$P$p$NWAWS) for this problem, which ensures convergence and the descent property of the objective function. Extensive experiments on six real-world data sets demonstrate that $q$P$p$NWAWS successfully solves the singularity problem and achieves a linear computational convergence rate in practical scenarios.
\end{abstract}

%

\section{Introduction}\label{sec:intro}

The Weber location problem is a fundamental problem that is extensively investigated in artificial intelligence \cite{lai2024singularity}, machine learning \cite{olpsjmlr,SSPO,SPOLC}, financial engineering \cite{egrmvgap}, computer vision \cite{lqmean}, and operations research \cite{weiszconvOR}. For a general definition, it seeks a point $\bx_*$ that minimizes the weighted sum of the $q$-th power of the $\ell_p$ distances to $m$ fixed data points $\{\bx_i\}_{i=1}^m\subseteq \bbR^d$ \cite{l1median1,morris1981convergence,extendfermat2,brimberg1993global}, defined as the $q$-th-powered $\ell_p$-norm Weber location problem \textbf{($q$P$p$NWLP)}:
\begin{equation}
\setlength{\abovedisplayskip}{1pt}
\setlength{\belowdisplayskip}{1pt}
\label{mod:efw}
\bx_*\in \argmin_{\by\in\bbR^d} C_{p,q}(\by):=\sum_{i=1}^m \xi_i\|\by-\bx_i\|_p^q,
\end{equation}
where $\xi_i$ denotes the weight for the $i$-th data point, $\|\cdot\|_p$ denotes the $\ell_p$ norm, $C_{p,q}(\cdot)$ denotes the cost function related to the $q$-th power of the $\ell_p$ norm, $1\leqslant q\leqslant p$ and $p\geqs 1$. Due to the same reasons as \cite{lai2024singularity}, we can assume that the data points $\{\bx_i\}_{i=1}^m$ are distinct and non-collinear in the rest of the paper.

\subsection{The Singularity Problem}
\label{sec:singprob}
\renewcommand{\thefootnote}{\fnsymbol{footnote}}
\footnotetext[7]{The supplementary material and codes for this paper are available at \url{https://github.com/laizhr/qPpNWAWS}.}
There is no closed-form solution to \eqref{mod:efw}, and the gradient-type method is an intuitive and tractable approach. This approach includes those do not explicitly show a gradient form, like the Weiszfeld algorithm \cite{brimberg1993global}. To take a glimpse of the singularity problem, we first compute the gradient of the objective function $C_{p,q}$ as
\begin{align}
\label{eqn:costgrad}
&\big( \nabla C_{p,q}(\by) \big)^{(t)}\nonumber\\ 
{:=}&{\sum_{i=1}^m} q\xi_i\| \by-\bx_i  \|_p^{q-p}|{y}^{(t)}-{x}_i^{(t)}|^{p-2}(y^{(t)}-{x}_i^{(t)}), 
\end{align}
where $\big( \nabla C_{p,q}(\by) \big)^{(t)}$ denotes the $t$-th dimension ($1\leqs t \leqs d$) of the gradient. Note that if $q<p$ or $p<2$, the singularity problem can occur if $\by$ hits the following singular set: 
\begin{equation}\label{def:sin}
\begin{cases}
\mathcal{S}_p{:=}\bigl\{\by{\in}\bbR^d|\exists i{\in}\{1,\dots,m \}, t{\in}\{1,\dots,d\}\text{ s.t. } y^{(t)}{=}x_i^{(t)}\bigr\}\\
\qquad\qquad\qquad\qquad\qquad\qquad\text{if }1\leqs p<2,\\
\mathcal{S}_2{:=}\{\bx_i\}_{i=1}^m \qquad\qquad\qquad\text{if }1\leqs q<2, p=2.
\end{cases}	
\end{equation} 
This singularity problem occurs frequently and unexpectedly. \citet{openweiszfeld,l1median3} indicate that the ``bad'' points that can yield such singularity in a gradient-type method may constitute a continuum set that can be dense in an open region of $\mathbb{R}^d$ ($d\geqslant 2$) or even the entire $\mathbb{R}^d$. Moreover, this problem cannot be circumvented by straightforward treatments like perturbations and random restarts. More details can be found in \cite{lai2024singularity}. 

If $q\geqs p$ or $p\geqs 2$, the term $\| \by-\bx_i  \|_p^{q-p}$ or $|{y}^{(t)}-{x}_i^{(t)}|^{p-2}$ in \eqref{eqn:costgrad} will not be singular. For $p\geqs 2$ but $1\leqs q< p$, the singularity problem can be solved by \cite{lai2024singularity}. Hence this paper mainly focuses on the rest unsolved singular cases with the two key hyperparameters $1\leqs q\leqs p$ and $1\leqs p<2$. Users can select any values of $p$ and $q$ within this range to suit their specific needs. As $p$ approaches $2$, the distance between data points becomes closer to the Euclidean distance. Conversely, as $p$ approaches $1$, the distance between data points becomes closer to the Manhattan distance, representing a typical non-Euclidean geometry. On the other hand, as $q$ approaches $p$, the distance between data points gets higher power, which can be advantageous in certain computer vision tasks \cite{lqmean}. In summary, allowing for a range of values for $p$ and $q$ enhances the geometrical representation of the Weber location problem.

\subsection{Continuum Singular Set for $1\leqs p<2$}
\label{sec:infsingset}
\eqref{def:sin} indicates that when $p=2$, the singularity only occurs at the $m$ fixed data points $\{\bx_i\}_{i=1}^m$. This case has already been completely solved by \cite{l1median3,lai2024singularity}. However, \textbf{when $1\leqs p<2$, the singularity occurs in at most $m\cdot d$ hyperplanes that encompass a continuum set of points, which is much more difficult than the $p=2$ case and no solutions have been developed (see Figure \ref{fig:sing1sing2})}. The difficulties lie in the following three aspects:
\begin{enumerate}[leftmargin=*]
\item An effective gradient-type algorithm may visit each singular point for only once. But since there are infinite singular points when $1\leqs p<2$, the algorithm may visit the singular set for infinite times.
\item Based on the first point, there is no unified step size for the escape from the singular set, which affects the convergence property.
\item The geometry of the objective function is complicated when $1\leqslant q\leqs p$ and $1\leqs p<2$ so that the characterizations of the subgradients, minimum and descent direction are very difficult, which also affects the convergence property.
\end{enumerate}

\begin{figure}[!htb]
\centering
\subfloat[$p=2$ and $1\leqs q<2$]{
\centering
\includegraphics[width=0.43\columnwidth]{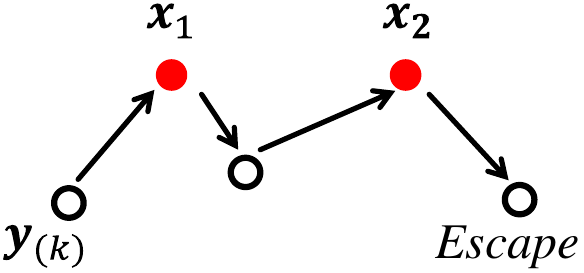}\label{fig:sing1}}
\subfloat[$1\leqs p<2$]{
\centering
\includegraphics[width=0.54\columnwidth]{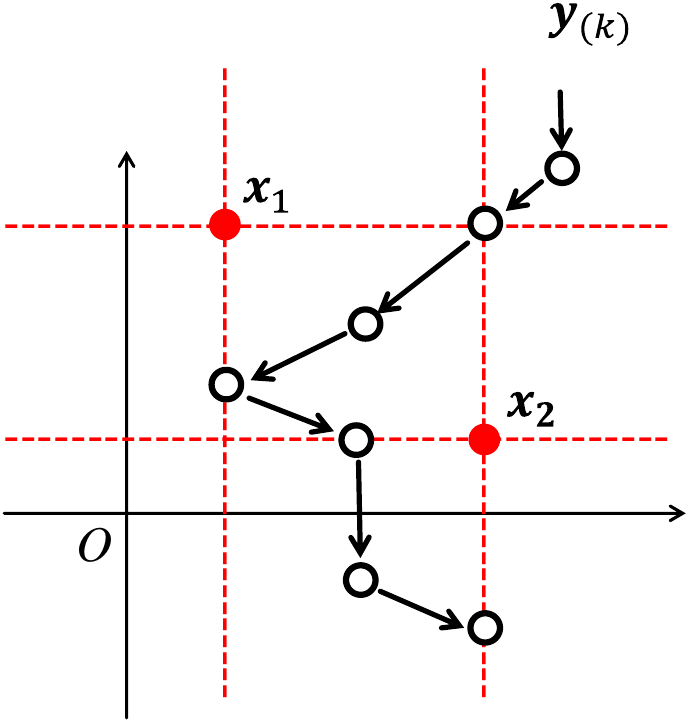}\label{fig:sing2}}
\caption{(a) When $p=2$, the singular set $\mathcal{S}_2$ (the red dots) is finite and an effective gradient-type algorithm visits each singular point for only once. Hence the iterate $\by_{(k)}$ (the circle) can finally escape from all the singular points. (b) When $1\leqs p<2$, the singular set $\mathcal{S}_p$ is a continuum (the red dashed lines), hence the iterate $\by_{(k)}$ may revisit $\mathcal{S}_p$ for infinite times and may not escape from $\mathcal{S}_p$.}
\label{fig:sing1sing2}
\end{figure}

To address the above difficulties, we develop a complete de-singularity subgradient methodology for $q$P$p$NWLP with $1\leqs q\leqs p$ and $1\leqs p<2$, including all the rest unsolved situations in this problem. \textbf{Our main contributions are summarized as follows.}
\begin{enumerate}[leftmargin=*]
\item We develop a de-singularity subgradient of the cost function $C_{p,q}$ on the singular set $\mathcal{S}_p$. It can replace the ordinary gradient without increasing computational complexity.
\item We develop a $q$-th-Powered $\ell_p$-Norm Weiszfeld Algorithm without Singularity (\textbf{$q$P$p$NWAWS}). It can identify whether the current iterate is a minimum point; If not, it can further reduce the cost function, no matter whether this iterate is singular or nonsingular. By this way, $q$P$p$NWAWS solves the singularity problem.
\item We develop a complete proof for the convergence of $q$P$p$NWAWS.
\item We demonstrate that $q$P$p$NWAWS achieves a linear computational convergence rate in practical scenarios.
\end{enumerate}

\section{Related Works}
We review some closely related works on $q$P$p$NWLP in this section.

\subsection{The $\ell_p$-Norm Weiszfeld Algorithm}
\label{sec:lqIRLSinterpret}
The $\ell_p$-norm Weiszfeld algorithm ($p$NWA) can be derived by the first-order optimal condition of the $q$P$p$NWLP with $1\leqs p\leqs 2$ and $q=1$ \cite{brimberg1993global}. For a \textbf{nonsingular optimal point} $\bx_*\notin\mathcal{S}_p$, setting $\big( \nabla C_{p,1}(\bx_*) \big)^{(t)}=0$ in \eqref{eqn:costgrad} yields 
\begin{equation}
\setlength{\abovedisplayskip}{1pt}
\setlength{\belowdisplayskip}{1pt}
\label{eq:ystar1}
x_*^{(t)}{=}\frac{{\sum_{i=1}^m} \xi_i\| \bx_*{-}\bx_i  \|_p^{1-p}|x_*^{(t)}{-}{x}_i^{(t)}|^{p-2}{x}_i^{(t)}}{\sum_{i=1}^m \xi_i\| \bx_*{-}\bx_i  \|_p^{1-p}|x_*^{(t)}{-}{x}_i^{(t)}|^{p-2}}, \forall 1{\leqs} t{\leqs} d.
\end{equation} 
These are fixed-point equations, which can be converted to a fixed-point algorithm with the following operator $\mathbf{T}_{p,1}$:
\begin{align}
\label{eqn:lqwa3}
y_{(k+1)}^{(t)}{:=}&  \big(\mathbf{T}_{p,1}(\by_{(k)})\big)^{(t)}\nonumber\\{:=}&\frac{{\sum_{i=1}^m} \xi_i\| \by_{(k)}-\bx_i  \|_p^{1-p}|y_{(k)}^{(t)}-{x}_i^{(t)}|^{p-2}{x}_i^{(t)}}{\sum_{i=1}^m \xi_i\| \by_{(k)}-\bx_i  \|_p^{1-p}|y_{(k)}^{(t)}-{x}_i^{(t)}|^{p-2}},
\end{align}
where $y_{(k)}^{(t)}$ denotes the $t$-th dimension of the $k$-th iterate. \textbf{When $\by_{(k)}$ hits the singular set $\mathcal{S}_p$, $p$NWA cannot solve it but just terminates at $\by_{(k)}$.}

\subsection{$q$-th Power Weiszfeld Algorithm without Singularity}
\citet{lai2024singularity} propose a $q$-th power Weiszfeld algorithm without singularity ($q$PWAWS) to handle the singularity problem in the special case $p=2, 1\leqs q<2$:
\begin{equation}
\label{eqn:lqwagen1}
\by_{(k+1)}{=}
\begin{cases}
\mathbf{T}_{2,q}( \by_{(k)}):=\frac{\sum_{i=1}^m \xi_i\| \by_{(k)}-\bx_i  \|_2^{q-2}\bx_i}{\sum_{i=1}^m \xi_i\| \by_{(k)}-\bx_i  \|_2^{q-2}}\\
\qquad\qquad\qquad\qquad\text{ if }\by_{(k)}{\notin} \{\bx_i\}_{i=1}^m,\\
\mathbf{T}_s( \by_{(k)}):=\by_{(k)}{-}\lambda_* \nabla D_{2,q}(\by_{(k)})\\
\quad\text{ if }\by_{(k)}{=}\bx_l\text{ for some }\ l{\in}\{1,\dots,m\},
\end{cases}
\end{equation}
where 
\begin{equation}
\setlength{\abovedisplayskip}{1pt}
\setlength{\belowdisplayskip}{1pt}
\label{def_desingp2}
\nabla D_{2,q}(\by_{(k)})=\sum_{i\ne l} q\xi_i\| \by_{(k)}-
\mathbf{x}_i  \|_2^{q-2}(\by_{(k)}-\mathbf{x}_i)
\end{equation}
is the $q$-th-powered $\ell_2$-norm de-singularity subgradient and $\lambda_*>0$. Intuitively, $\nabla D_{2,q}$ removes the singular component corresponding to $\bx_l$ and serves as a conventional gradient in the subgradient descent step $\mathbf{T}_s$.

$q$PWAWS can escape from each singular point and will not revisit it again. Since there are finite singular points when $p=2$, $q$PWAWS can eventually escape from the singular set $\mathcal{S}_2$ (see Figure \ref{fig:sing1}). \textbf{However, it cannot be directly adopted in the $1\leqs p< 2$ case, because the singular set $\mathcal{S}_p$ is a continuum and $q$PWAWS may not eventually escape from $\mathcal{S}_p$. Moreover, the step size $\lambda_*$ in $\mathbf{T}_s$ cannot be uniformly chosen with respect to (w.r.t.) $\mathcal{S}_p$ due to its infinite elements. The characterizations of the subgradients, minimum and descent direction are also difficult. These three problems affect the convergence property of $q$PWAWS.} Therefore, a new methodology should be developed to overcome these new challenges for the $1\leqs p< 2$ case.

\section{$q$-th-Powered $\ell_p$-Norm Weiszfeld Algorithm without Singularity}
In this section, we present $q$P$p$NWAWS for solving $q$P$p$NWLP with $1\leqs p<2$ and $1\leqs q\leqs p$. For a convenient illustration, we let $\eta_i:={\xi_i}^{\frac{1}{q}}$ and reformulate \eqref{mod:efw} as
\begin{equation}
\setlength{\abovedisplayskip}{1pt}
\setlength{\belowdisplayskip}{1pt}
\label{eqn:lpqmediangenreal}
\bx_*\in \argmin_{\by\in\bbR^d} C_{p,q}(\by):=\sum_{i=1}^m \eta_i^q\|\by-\bx_i\|_p^q.
\end{equation}
$C_{p,q}$ is strictly convex and there exists a unique minimum point $\bx_*$ for \eqref{eqn:lpqmediangenreal} with $1<p<2$, while $C_{p,q}$ is convex and there exists one or more minimum points for \eqref{eqn:lpqmediangenreal} with $p=1$. The construction of $q$P$p$NWAWS consists of $4$ steps: 
\begin{enumerate}
\item The update at a nonsingular iterate $\by_{(k)}\notin \mathcal{S}_p$ is designed as a Weiszfeld-style one, which makes $C_{p,q}(\by_{(k)})$ non-increasing.
\item Define the $q$-th-powered $\ell_p$-norm de-singularity subgradient $\nabla D_{p,q}(\by)$ for $\by\in\mathcal{S}_p$ to characterize the subgradients of $C_{p,q}(\by)$ and the minimum of $C_{p,q}$.
\item Adopt $\nabla D_{p,q}(\by_{(k)})$ to construct the iterative update at the singular iterate $\by_{(k)}\in\mathcal{S}_p$ to reduce the cost function.
\item Establish the convergence proof for $q$P$p$NWAWS.
\end{enumerate}

\subsection{Update Formula at Nonsingular Iterates}
\label{sec:genupdatewc}
First, we consider the simplest case where the current iterate $\by_{(k)}\notin \mathcal{S}_p$. The corresponding update formula is:
\begin{align}
\label{eqn:lqwaeta}
y_{(k+1)}^{(t)}&:= \big(\mathbf{T}_{p,q}(\by_{(k)})\big)^{(t)}\nonumber\\
&{:=}\frac{\sum_{i=1}^m \eta_i^q\| \by_{(k)}{-}\bx_i  \|_p^{q{-}p}|{y}_{(k)}^{(t)}{-}{x}_i^{(t)}|^{p{-}2}{x}_i^{(t)}}{\sum_{i=1}^m \eta_i^q\| \by_{(k)}{-}\bx_i  \|_p^{q{-}p}|{y}_{(k)}^{(t)}{-}{x}_i^{(t)}|^{p{-}2}}.
\end{align}
The following descent property guarantees that \eqref{eqn:lqwaeta} will reduce $C_{p,q}$ at any non-minimum nonsingular iterate.
\begin{theorem}[Descent Property at Nonsingular Iterates]
\label{thm:nonincreasing}
Let the cost function $C_{p,q}$ and the operator $\mathbf{T}_{p,q}$ be defined in \eqref{eqn:lpqmediangenreal} and \eqref{eqn:lqwaeta}, respectively. For $1\leqs p<2$ and $1\leqs q\leqs p$, if $\by_{(k)}\notin \mathcal{S}_p$, then $C_{p,q}(\mathbf{T}_{p,q}( \by_{(k)}))\leqs C_{p,q}(\by_{(k)})$ with equality holds only when $\mathbf{T}_{p,q}( \by_{(k)})= \by_{(k)}$.
\end{theorem}
The proof is provided in Supplementary \ref{sm_thm:nonincreasing}. The following corollary characterizes a minimum point $\bx_*$ of \eqref{eqn:lpqmediangenreal} if $\bx_*\notin \mathcal{S}_p$.
\begin{corollary}
\label{cor:charnonsing}
If $\by_{(k)}\notin\mathcal{S}_p$, then $\mathbf{T}_{p,q}(\by_{(k)})= \by_{(k)} \Leftrightarrow$ $\by_{(k)}$ is a minimum point of model \eqref{eqn:lpqmediangenreal}, i.e., $\by_{(k)}=\bx_*$.
\end{corollary}	
The proof is provided in Supplementary \ref{sm_{cor:charnonsing}}. We then turn to the singular case.

\subsection{Characterization of Subgradients and Minimum}
Before we derive the iterative update for the singular iterate $\by_{(k)}\in\mathcal{S}_p$, we first introduce the de-singularity subgradient of $C_{p,q}(\by_{(k)})$ and characterize the minimum point(s) of \eqref{eqn:lpqmediangenreal}. 
\begin{definition}[Subgradient, \citealt{rockafellar2009variational}]
\label{def:convsubdifferential}
Let $C_{p,q}:\bbR^d\to\bbR$ be a convex function. A vector $\bv\in\bbR^d$ is called a subgradient of $C_{p,q}$ at $\by\in\bbR^d$ if for all $\bz{\in}\mathbb{R}^d$, 
\begin{equation}
\label{eqn:convsubdifferential}
C_{p,q}(\bz){-}C_{p,q}(\by){\geqs}\bm{v}^{\top}(\bm{z}-\bm{y}).
\end{equation}
The set of all subgradients at $\by$ is denoted by $\partial C_{p,q}(\by)$. If $C_{p,q}$ is differentiable at $\by$, then $\partial C_{p,q}(\by)$ reduces to the gradient $\nabla C_{p,q}(\by)$. 
\end{definition}
To construct $\partial C_{p,q}(\by_{(k)})$, we need to identify the singular component(s) of $\by_{(k)}$.

\begin{definition}[Singular Component(s)]
For $\by\in\mathcal{S}_p$, each $t\in\{1,\dots,d\}$ and each $i\in\{1,2,\dots,m\}$, let
\begin{align}
\label{def:ui}	
U_i(\by)&:=\{t\in\{1,\dots,d\}:\ y^{(t)}=x_i^{(t)} \},\\
\label{def:vt}
V_t(\by)&:=\{i\in\{1,\dots,m\}:\ y^{(t)}=x_i^{(t)} \},
\end{align}
which represent the index sets of the dimensions and the data points such that ${y}^{(t)}={x}_i^{(t)}$, respectively. 
\end{definition}
Figure \ref{fig:uivt} shows an intuitive example for $U_i(\by_{(k)})$ and $V_t(\by_{(k)})$.

\begin{figure}[!htb]
\centering
\includegraphics[width=0.76\columnwidth]{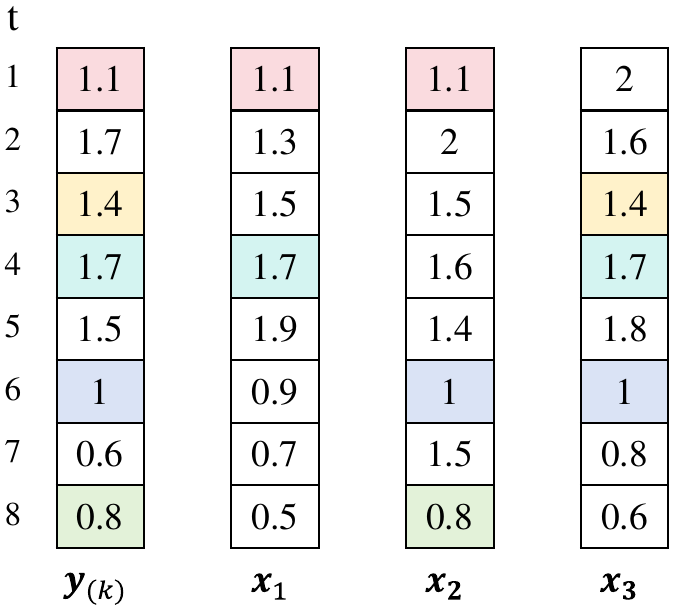}
\caption{An intuitive example for $U_i(\by_{(k)})$ and $V_t(\by_{(k)})$: $U_1(\by_{(k)})=\{1,4\}$, $U_2(\by_{(k)})=\{1,6,8\}$, $U_3(\by_{(k)})=\{3,4,6\}$, and $V_1(\by_{(k)})=\{1,2\}$, $V_2(\by_{(k)})=\emptyset$, $V_3(\by_{(k)})=\{3\}$, $V_4(\by_{(k)})=\{1,3\}$, $V_6(\by_{(k)})=\{2,3\}$, $V_8(\by_{(k)})=\{2\}$.}
\label{fig:uivt}
\end{figure}

\begin{definition}[$\mathbf{q}$\textbf{-th-Powered $\ell_p$-Norm De-singularity Subgradient}]
\label{defn:desinggrad}
By removing the singular term(s), we define the de-singularity part $D_{p,q}:\bbR^d\to\bbR$ of $C_{p,q}$  and the $q$-th-powered $\ell_p$-norm de-singularity subgradient as follows:
\begin{align}
D_{p,q}(\by)&:=\sum_{i=1}^m\eta_i^q\big(\sum_{t\notin U_i(\by)}| {y}^{(t)}-{x}_i^{(t)}|^p\big)^{\frac{q}{p}},\label{eqn:dq}	\\
\big(\nabla D_{p,q}(\by)& \big)^{(t)}  \nonumber\\
{:=}\sum_{i\notin V_t(\by)} q\eta_i^q&\| \by{-}\bx_i  \|_p^{q-p}|{y}^{(t)}{-}{x}_i^{(t)}|^{p-2}(y^{(t)}{-}{x}_i^{(t)}).\label{eqn:dqgrad}
\end{align}
\end{definition}

\begin{theorem}[Characterization of Subgradients and Minimum]
\label{thm:charsing}
For $\by\in\mathcal{S}_p$,
\begin{equation}
\setlength{\abovedisplayskip}{1pt}
\setlength{\belowdisplayskip}{1pt}
\label{eqn:subgrad}
\partial C_{p,q}(\by){=}
\begin{cases}
\{\nabla D_{1,1}(\by){+}\bu\}\ \text{where}\ -a^{(t)}{\leqs} u^{(t)}{\leqs} a^{(t)},\\
 \qquad \forall t,\qquad \text{if }p=q=1,\\
\{\nabla D_{p,1}(\by)\},\ \text{if }q{=}1,1{<}p{<}2, \\
\qquad\qquad\qquad\qquad\by\in \mathcal{S}_p\backslash\{\bx_i\}_{i=1}^m,\\
\{\nabla D_{p,1}(\bx_l)+\eta_l \bb\}\quad \text{where}\quad \|\bb\|_r\leqs 1,\\
\qquad\qquad\qquad \text{if }q{=}1,1{<}p{<}2, \by=\bx_l,\\
\{\nabla D_{p,q}(\by)\},\ \text{if }1<q\leqs p,1<p<2,
\end{cases}
\end{equation}
where $a^{(t)}=\sum_{i\in V_t(\by) }\eta_i$ and $\|\cdot\|_r$ is the conjugate norm of $\|\cdot\|_p$ such that $\frac{1}{r}+\frac{1}{p}=1$.
\end{theorem}
The proof is provided in Supplementary \ref{sm_thm:charsing}. According to Fermat's rule, $\by\in\mathcal{S}_p$ is a minimum point of \eqref{eqn:lpqmediangenreal} if and only if $\mathbf{0}_d\in \partial C_{p,q}(\by)$, which is easy to verify. If $\by_{(k)}$ is not a minimum point, the following theorem shows a descent direction of $C_{p,q}(\by_{(k)})$.

\begin{theorem}[Descent Property at Singular Iterates]
\label{thm:singmin}
For $1\leqs p<2$ and $1\leqs q\leqs p$, define the following direction
\begin{equation}
\setlength{\abovedisplayskip}{1pt}
\setlength{\belowdisplayskip}{1pt}
\label{eqn:descdirect}
\sD_{p,q}(\by){=}
\begin{cases}
\big(\nabla D_{p,1}(\bx_l)\big)^{\frac{r}{p}}\ \ \text{where }\ \frac{1}{r}+\frac{1}{p}=1,\      \\
\qquad\qquad  \ \text{if }q{=}1,1{<}p{<}2, \by=\bx_l,\\
\nabla D_{p,q}(\by),\qquad\qquad \text{else },
\end{cases}
\end{equation}
where $(\cdot)^{\frac{r}{p}}$ denotes the element-wise signed power. If $\by\in\mathcal{S}_p$ is not a minimum point of \eqref{eqn:lpqmediangenreal}, then there exists some $\lambda_*>0$ such that for any $0<\lambda\leqs\lambda_*$, $C_{p,q}(\by-\lambda \sD_{p,q}(\by))<C_{p,q}(\by)$.
\end{theorem}

The proof is provided in Supplementary \ref{sm_thm:singmin}. The key point is to verify that $-\sD_{p,q}(\by)$ is a descent direction if $\by$ is not a minimum point. To determine the step size $\lambda_*$ in practice, we can start with an initial value of $\lambda_0=\|\sD_{p,q}(\by_{(k)})\|_p$ and implement a line search $\lambda_{w+1}=\rho\lambda_w$ with $0<\rho<1$, until we find a value of $\lambda_*$ such that $C_{p,q}(\by_{(k)}-\lambda_* \sD_{p,q}(\by_{(k)}))< C_{p,q}(\by_{(k)})$.  Then we can construct the iterative update at the singular point $\by_{(k)}$ as
\begin{equation}
\setlength{\abovedisplayskip}{1pt}
\setlength{\belowdisplayskip}{1pt}
\label{def:ts}
\by_{(k+1)}:=\mathbf{T}_s( \by_{(k)}):=\by_{(k)}-\lambda_* \sD_{p,q}(\by_{(k)}).	
\end{equation}
Combining \eqref{eqn:lqwaeta} and \eqref{def:ts}, the whole $q$P$p$NWAWS is given by:
\begin{equation}
\setlength{\abovedisplayskip}{1pt}
\setlength{\belowdisplayskip}{1pt}
\label{eqn:lqwagen}
\by_{(k+1)}{:=}\mathbf{T}( \by_{(k)}){:=}
\begin{cases}
\mathbf{T}_{p,q}(\by_{(k)})&
\text{if}\: \by_{(k)}\notin \mathcal{S}_p,\\
\mathbf{T}_s( \by_{(k)})&\text{if}\: \by_{(k)}\in \mathcal{S}_p.	
\end{cases}	
\end{equation}

Theorems \ref{thm:nonincreasing} and \ref{thm:singmin} indicate that $C_{p,q}(\by_{(k+1)})< C_{p,q}(\by_{(k)})$ if $\by_{(k)}$ is not a minimum point (which can be characterized by $q$P$p$NWAWS). Moreover, since $C_{p,q}(\by)\geqs 0$ for all $\by\in\bbR^d$, we can conclude that the sequence $\{C_{p,q}(\by_{(k)})\}_{k\in\bbN}$ converges.

\begin{table*}[t]
\centering
\scalebox{0.88}{\begin{tabular}{cccccc}
\toprule
Data Set   & Region & Time         &        Periods          & Frequency & \# Assets \\ \hline
CSI300  & CN     & Mar/16/2015- May/19/2017& 534    &  Daily   & 47     \\
NYSE(N)  & US     &  Jan/1/1985 - Jun/30/2010 & 6431    &  Daily   & 23     \\
FTSE100  & UK     & Nov/07/2002 - Nov/04/2016 & 717    &  Weekly   & 83     \\
NASDAQ100      & US     & Mar/11/2004 - Nov/04/2016 & 596    &  Weekly   & 82     \\
FF100 & US     & Jul/1971 - May/2023 &  623  & Monthly   & 100    \\
FF100MEOP & US     & Jul/1971 - May/2023 & 623  & Monthly   & 100    \\ \bottomrule
\end{tabular}}
\caption{Profiles of six benchmark data sets.}
\label{tab:datasetsprofiles}
\end{table*}

\begin{table*}[h]
	\centering	
	\scalebox{0.78}{
		\begin{tabular}{c c c c c c c c c c c }
			\hline
			\diagbox{q}{p}& 1.0& 1.1 & 1.2 & 1.3 & 1.4 & 1.5 & 1.6 & 1.7 & 1.8 & 1.9 \\ \hline
			1.0 & $4.92\pm 0.29$ &	$2.56 \pm1.07 $&	$2.15 \pm0.42 $&	$2.08 \pm0.36 $&	$2.04 \pm0.30 $&	$2.02 \pm0.28 $&	$1.99 \pm0.27 $&	$1.98 \pm0.25 $&	$1.97 \pm0.24 $&	$1.97 \pm0.24 $\\
			1.1 &-&	$2.83 \pm1.27 $&	$2.11 \pm0.39 $&	$2.07 \pm0.33 $&	$2.02 \pm0.29 $&	$2.00 \pm0.26 $&	$1.99 \pm0.26 $&	$1.97 \pm0.24 $&	$1.97 \pm0.24 $&	$1.96 \pm0.23 $\\
			1.2 &-&	-&	$2.24 \pm0.63 $&	$2.04 \pm0.30 $&	$2.01 \pm0.27 $&	$2.00 \pm0.25 $&	$1.98 \pm0.24 $&	$1.97 \pm0.23 $&	$1.96 \pm0.23 $&	$1.96 \pm0.23 $\\
			1.3 &-&	-&	-&	$2.05 \pm0.36 $&	$2.00 \pm0.26 $&	$1.98 \pm0.24 $&	$1.97 \pm0.22 $&	$1.96 \pm0.22 $&	$1.96 \pm0.23 $&	$1.96 \pm0.23 $\\
			1.4 &-&	-&	-&	-&	$2.00 \pm0.25 $&	$1.98 \pm0.23 $&	$1.96 \pm0.22 $&	$1.96 \pm0.22 $&	$1.96 \pm0.23 $&	$1.95 \pm0.23 $\\
			1.5 &-&	-&	-&	-&	-&	$1.98 \pm0.22 $&	$1.96 \pm0.23 $&	$1.96 \pm0.23 $&	$1.96 \pm0.22 $&	$1.95 \pm0.22 $\\
			1.6 &-&	-&	-&	-&	-&	-&	$1.96 \pm0.23 $&	$1.96 \pm0.22 $&	$1.95 \pm0.22 $&	$1.95 \pm0.22 $\\
			1.7 &-&	-&	-&	-&	-&	-&	-&	$1.96 \pm0.21 $&	$1.95 \pm0.22 $&	$1.95 \pm0.22 $\\
			1.8 &-&	-&	-&	-&	-&	-&	-&	-&	$1.95 \pm0.22 $&	$1.95 \pm0.22 $\\
			1.9 &-&	-&	-&	-&	-&	-&	-&	-&	-&	$1.95 \pm0.22 $\\
			\hline
	\end{tabular}}
	\caption{Average number of iterates for $q$P$p$NWAWS to reduce the cost function at a singular point on CSI300 (mean$\pm$STD).}
	\label{table1}
\end{table*}

\begin{table*}[!h]
	\centering	
	\scalebox{0.78}{
		\begin{tabular}{c c c c c c c c c c c }
			\hline
			\diagbox{q}{p}& 1.0& 1.1 & 1.2 & 1.3 & 1.4 & 1.5 & 1.6 & 1.7 & 1.8 & 1.9 \\ \hline
			1.0 &$3.91\pm0.60$&	$1.90 \pm0.98 $&	$1.74 \pm0.81 $&	$1.68 \pm0.78 $&	$1.64 \pm0.74 $&	$1.60 \pm0.68 $&	$1.57 \pm0.64 $&	$1.54 \pm0.61 $&	$1.53 \pm0.60 $&	$1.52 \pm0.59 $\\
			1.1 &-&	$2.07 \pm1.16 $&	$1.72 \pm0.81 $&	$1.67 \pm0.78 $&	$1.63 \pm0.73 $&	$1.60 \pm0.70 $&	$1.57 \pm0.66 $&	$1.54 \pm0.61 $&	$1.53 \pm0.60 $&	$1.52 \pm0.59 $\\
			1.2 &-&	-&	$1.75 \pm0.86 $&	$1.66 \pm0.77 $&	$1.62 \pm0.72 $&	$1.59 \pm0.68 $&	$1.56 \pm0.64 $&	$1.53 \pm0.61 $&	$1.52 \pm0.59 $&	$1.51 \pm0.58 $\\
			1.3 &-&	-&	-&	$1.65 \pm0.77 $&	$1.61 \pm0.70 $&	$1.58 \pm0.66 $&	$1.55 \pm0.63 $&	$1.53 \pm0.60 $&	$1.51 \pm0.58 $&	$1.50 \pm0.57 $\\
			1.4 &-&	-&	-&	-&	$1.60 \pm0.68 $&	$1.57 \pm0.66 $&	$1.55 \pm0.63 $&	$1.52 \pm0.59 $&	$1.50 \pm0.57 $&	$1.49 \pm0.55 $\\
			1.5 &-&	-&	-&	-&	-&	$1.57 \pm0.65 $&	$1.55 \pm0.63 $&	$1.51 \pm0.58 $&	$1.49 \pm0.55 $&	$1.47 \pm0.53 $\\
			1.6 &-&	-&	-&	-&	-&	-&	$1.54 \pm0.63 $&	$1.50 \pm0.57 $&	$1.47 \pm0.53 $&	$1.46 \pm0.51 $\\
			1.7 &-&	-&	-&	-&	-&	-&	-&	$1.48 \pm0.54 $&	$1.46 \pm0.50 $&	$1.45 \pm0.50 $\\
			1.8 &-&	-&	-&	-&	-&	-&	-&	-&	$1.46 \pm0.50 $&	$1.45 \pm0.50 $\\
			1.9 &-&	-&	-&	-&	-&	-&	-&	-&	-&	$1.45 \pm0.50 $\\
			\hline
	\end{tabular}}
	\caption{Average number of iterates for $q$P$p$NWAWS to reduce the cost function at a singular point on NYSE(N) (mean$\pm$STD).}
	\label{table11}
\end{table*}

\begin{table*}[!htb]
\centering	
\scalebox{0.68}{
\begin{tabular}{cccccccccccc}
\hline
\diagbox{q}{p}& &1.0 &1.1 & 1.2 & 1.3 & 1.4 & 1.5 & 1.6 & 1.7 & 1.8 & 1.9 \\ \hline
\multirow{2}{*}{1.0} & $Time$ &$0.0271$&	$0.0193 $&	$0.0180 $&	$0.0169 $&	$0.0161 $&	$0.0155 $&	$0.0148 $&	$0.0143 $&	$0.0135 $&	$0.0130 $\\
&$Iter$ &$35.93\pm15.04$&	$15.02 \pm2.79 $&	$13.89 \pm2.07 $&	$13.38 \pm2.49 $&	$12.83 \pm2.68 $&	$12.34 \pm2.78 $&	$11.90 \pm2.82 $&	$11.51 \pm2.83 $&	$11.19 \pm2.81 $&	$10.92 \pm2.83 $\\	 \hline
\multirow{2}{*}{1.1} & $Time$ &-&	$0.0184 $&	$0.0173 $&	$0.0162 $&	$0.0153 $&	$0.0144 $&	$0.0136 $&	$0.0131 $&	$0.0123 $&	$0.0118 $\\
&$Iter$ &-&	$15.01 \pm2.37 $&	$13.63 \pm1.89 $&	$12.95 \pm2.14 $&	$12.28 \pm2.25 $&	$11.73 \pm2.35 $&	$11.27 \pm2.37 $&	$10.84 \pm2.37 $&	$10.51 \pm2.37 $&	$10.24 \pm2.36 $\\	
 \hline
\multirow{2}{*}{1.2} & $Time$ &-&	-&	$0.0170 $&	$0.0157 $&	$0.0145 $&	$0.0136 $&	$0.0127 $&	$0.0122 $&	$0.0114 $&	$0.0109 $\\
&$Iter$ &-&	-&	$13.65 \pm1.86 $&	$12.66 \pm1.89 $&	$11.86 \pm1.87 $&	$11.27 \pm1.96 $&	$10.75 \pm1.98 $&	$10.32 \pm2.01 $&	$10.01 \pm2.03 $&	$9.74 \pm2.03 $\\	
\hline
\multirow{2}{*}{1.3} & $Time$ &-&	-&	-&	$0.0151 $&	$0.0140 $&	$0.0128 $&	$0.0119 $&	$0.0114 $&	$0.0107 $&	$0.0102 $\\
&$Iter$ &-&	-&	-&	$12.45 \pm1.67 $&	$11.52 \pm1.55 $&	$10.82 \pm1.65 $&	$10.30 \pm1.64 $&	$9.90 \pm1.69 $&	$9.60 \pm1.74 $&	$9.32 \pm1.73 $\\	
\hline
\multirow{2}{*}{1.4} & $Time$ &-&	-&	-&	-&	$0.0135 $&	$0.0122 $&	$0.0113 $&	$0.0107 $&	$0.0100 $&	$0.0095 $\\
&$Iter$ &-&-&	-&	-&	$11.20 \pm1.24 $&	$10.44 \pm1.34 $&	$9.91 \pm1.38 $&	$9.48 \pm1.40 $&	$9.18 \pm1.48 $&	$8.90 \pm1.48 $\\	
\hline
\multirow{2}{*}{1.5} & $Time$ &-&	-&	-&	-&	-&	$0.0114 $&	$0.0109 $&	$0.0100 $&	$0.0092 $&	$0.0088 $\\
&$Iter$ &-&	-&	-&	-&	-&	$10.09 \pm1.04 $&	$9.55 \pm1.09 $&	$9.12 \pm1.15 $&	$8.75 \pm1.24 $&	$8.49 \pm1.27 $\\	
\hline
\multirow{2}{*}{1.6} & $Time$ &-&	-&	-&	-&	-&	-&	$0.0101 $&	$0.0094 $&	$0.0086 $&	$0.0081 $\\
&$Iter$ &-&	-&	-&	-&	-&	-&	$9.25 \pm0.86 $&	$8.77 \pm0.91 $&	$8.38 \pm1.05 $&	$8.08 \pm1.09 $\\	
\hline
\multirow{2}{*}{1.7} & $Time$ &-&	-&	-&	-&	-&	-&	-&	$0.0086 $&	$0.0079 $&	$0.0073 $\\
&$Iter$ &-&	-&	-&	-&	-&	-&	-&	$8.46 \pm0.73 $&	$7.97 \pm0.88 $&	$7.63 \pm1.02 $\\	\hline
\multirow{2}{*}{1.8} & $Time$ &-&	-&	-&	-&	-&	-&	-&	-&	$0.0072 $&	$0.0070 $\\
&$Iter$ &	-&-&	-&	-&	-&	-&	-&	-&	$7.66 \pm0.71 $&	$7.45 \pm0.69 $\\	
\hline
\multirow{2}{*}{1.9} & $Time$ &-&	-&	-&	-&	-&	-&	-&	-&	-&	$0.0064 $\\
&$Iter$ &	-&-&	-&	-&	-&	-&	-&	-&	-&	$7.19 \pm0.41 $\\	
\hline
\end{tabular}}
\caption{Average computational time (in seconds) and average number of iterations (mean$\pm$STD) for $q$P$p$NWAWS on CSI300.}
\label{tab:comcostcsi}
\end{table*}

\begin{table*}[!htb]
	\centering	
	\scalebox{0.695}{
		\begin{tabular}{cccccccccccc}
			\hline
			\diagbox{q}{p}& & 1.0 &1.1 & 1.2 & 1.3 & 1.4 & 1.5 & 1.6 & 1.7 & 1.8 & 1.9 \\ \hline
			\multirow{2}{*}{1.0} & $Time$ &	$0.0011$&$0.0050 $&	$0.0071 $&	$0.0074 $&	$0.0060 $&	$0.0064 $&	$0.0055 $&	$0.0035 $&	$0.0033 $&	$0.0032 $\\
			&$Iter$ &$26.71\pm15.93$&$14.84 \pm6.67 $&	$13.62 \pm4.51 $&	$12.72 \pm3.35 $&	$12.07 \pm2.86 $&	$11.50 \pm2.73 $&	$11.02 \pm2.54 $&	$10.62 \pm2.45 $&	$10.29 \pm2.44 $&	$10.03 \pm2.43 $\\
			\hline
			\multirow{2}{*}{1.1} & $Time$ &-&	$0.0067 $&	$0.0066 $&	$0.0068 $&	$0.0052 $&	$0.0061 $&	$0.0051 $&	$0.0032 $&	$0.0031 $&	$0.0029 $\\
			&$Iter$ &-&	$14.70 \pm8.64 $&	$13.29 \pm4.22 $&	$12.29 \pm3.04 $&	$11.56 \pm2.55 $&	$10.99 \pm2.34 $&	$10.52 \pm2.22 $&	$10.09 \pm2.11 $&	$9.77 \pm2.11 $&	$9.52 \pm2.10 $\\
			\hline
			\multirow{2}{*}{1.2} & $Time$ &-&	-&	$0.0075 $&	$0.0059 $&	$0.0044 $&	$0.0058 $&	$0.0048 $&	$0.0030 $&	$0.0029 $&	$0.0027 $\\
			&$Iter$ &-&	-&	$13.05 \pm3.99 $&	$11.96 \pm2.84 $&	$11.13 \pm2.27 $&	$10.53 \pm2.06 $&	$10.06 \pm1.95 $&	$9.64 \pm1.85 $&	$9.33 \pm1.87 $&	$9.08 \pm1.90 $\\
			\hline
			\multirow{2}{*}{1.3} & $Time$ &-&	-&	-&	$0.0056 $&	$0.0051 $&	$0.0057 $&	$0.0041 $&	$0.0029 $&	$0.0027 $&	$0.0026 $\\
			&$Iter$ &-&	-&	-&	$11.70 \pm2.75 $&	$10.77 \pm1.98 $&	$10.13 \pm1.80 $&	$9.65 \pm1.70 $&	$9.24 \pm1.64 $&	$8.92 \pm1.68 $&	$8.67 \pm1.72 $\\
			\hline
			\multirow{2}{*}{1.4} & $Time$ &-&	-&	-&	-&	$0.0048 $&	$0.0052 $&	$0.0028 $&	$0.0027 $&	$0.0025 $&	$0.0024 $\\
			&$Iter$ &-&-&	-&	-&	$10.47 \pm1.85 $&	$9.78 \pm1.58 $&	$9.29 \pm1.54 $&	$8.86 \pm1.47 $&	$8.55 \pm1.53 $&	$8.29 \pm1.57 $\\
			\hline
			\multirow{2}{*}{1.5} & $Time$ &-&	-&	-&	-&	-&	$0.0043 $&	$0.0026 $&	$0.0025 $&	$0.0024 $&	$0.0023 $\\
			&$Iter$ &-&	-&	-&	-&	-&	$9.47 \pm1.40 $&	$8.96 \pm1.43 $&	$8.52 \pm1.36 $&	$8.17 \pm1.40 $&	$7.89 \pm1.46 $\\	
			\hline
			\multirow{2}{*}{1.6} & $Time$ &-&	-&	-&	-&	-&	-&	$0.0025 $&	$0.0024 $&	$0.0022 $&	$0.0021 $\\
			&$Iter$ &-&	-&	-&	-&	-&	-&	$8.66 \pm1.37 $&	$8.17 \pm1.27 $&	$7.80 \pm1.34 $&	$7.52 \pm1.33 $\\	
			\hline
			\multirow{2}{*}{1.7} & $Time$ &-&	-&	-&	-&	-&	-&	-&	$0.0022 $&	$0.0020 $&	$0.0019 $\\
			&$Iter$ &-&	-&	-&	-&	-&	-&	-&	$7.83 \pm1.16 $&	$7.38 \pm1.02 $&	$7.10 \pm1.10 $\\
			\hline
			\multirow{2}{*}{1.8} & $Time$ &-&	-&	-&	-&	-&	-&	-&	-&	$0.0019 $&	$0.0018 $\\
			&$Iter$ &-&	-&	-&	-&	-&	-&	-&	-&	$7.07 \pm0.97 $&	$6.84 \pm1.03 $\\
			\hline
			\multirow{2}{*}{1.9} & $Time$ &-&-&	-&	-&	-&	-&	-&	-&	-&	$0.0016 $\\
			&$Iter$ &-&	-&	-&	-&	-&	-&	-&	-&	-&	$6.50 \pm0.89 $\\
			\hline
	\end{tabular}}
	\caption{Average computational time (in seconds) and average number of iterations (mean$\pm$STD) for $q$P$p$NWAWS on NYSE(N).}
	\label{tab:comcostnyse}
\end{table*}

\begin{table*}[t]
	\centering	
	\scalebox{0.8}{
		\begin{tabular}{c c c c c c c c c c c}
			\hline
			\diagbox{q}{p}& 1.0& 1.1 & 1.2 & 1.3 & 1.4 & 1.5 & 1.6 & 1.7 & 1.8 & 1.9 \\ \hline
1.0 &	$0.80 \pm0.07 $&	$0.64 \pm0.06 $&	$0.61 \pm0.06 $&	$0.58 \pm0.07 $&	$0.55 \pm0.08 $&	$0.52 \pm0.09 $&	$0.49 \pm0.10 $&	$0.47 \pm0.10 $&	$0.46 \pm0.11 $&	$0.44 \pm0.12 $\\	
1.1 &	-&	$0.63 \pm0.05 $&	$0.60 \pm0.05 $&	$0.56 \pm0.06 $&	$0.52 \pm0.06 $&	$0.49 \pm0.07 $&	$0.46 \pm0.08 $&	$0.44 \pm0.09 $&	$0.42 \pm0.10 $&	$0.40 \pm0.11 $\\	
1.2 &	-&	-&	$0.59 \pm0.05 $&	$0.55 \pm0.05 $&	$0.51 \pm0.05 $&	$0.47 \pm0.06 $&	$0.44 \pm0.07 $&	$0.41 \pm0.08 $&	$0.39 \pm0.09 $&	$0.37 \pm0.10 $\\	
1.3 &	-&	-&	-&	$0.54 \pm0.05 $&	$0.50 \pm0.04 $&	$0.45 \pm0.05 $&	$0.41 \pm0.06 $&	$0.38 \pm0.07 $&	$0.36 \pm0.08 $&	$0.34 \pm0.09 $\\	
1.4 &	-&	-&	-&	-&	$0.48 \pm0.04 $&	$0.43 \pm0.04 $&	$0.39 \pm0.05 $&	$0.35 \pm0.06 $&	$0.32 \pm0.07 $&	$0.30 \pm0.08 $\\	
1.5 &	-&	-&	-&	-&	-&	$0.41 \pm0.03 $&	$0.36 \pm0.04 $&	$0.32 \pm0.05 $&	$0.29 \pm0.06 $&	$0.26 \pm0.07 $\\	
1.6 &	-&	-&	-&	-&	-&	-&	$0.34 \pm0.03 $&	$0.29 \pm0.03 $&	$0.25 \pm0.05 $&	$0.22 \pm0.08 $\\	
1.7 &	-&	-&	-&	-&	-&	-&	-&	$0.26 \pm0.03 $&	$0.21 \pm0.04 $&	$0.17 \pm0.04 $\\	
1.8 &	-&	-&	-&	-&	-&	-&	-&	-&	$0.18 \pm0.02 $&	$0.14 \pm0.02 $\\	
1.9 &	-&	-&	-&	-&	-&	-&	-&	-&	-&	$0.09 \pm0.01 $\\	
			\hline
	\end{tabular}}
	\caption{Average computational convergence rate (mean$\pm$STD) for $q$P$p$NWAWS on CSI300.}
	\label{tab:converatecsi300}
\end{table*}

\begin{table*}[!h]
	\centering	
	\scalebox{0.8}{
		\begin{tabular}{c c c c c c c c c c c }
			\hline
			\diagbox{q}{p}& 1.0&1.1 & 1.2 & 1.3 & 1.4 & 1.5 & 1.6 & 1.7 & 1.8 & 1.9 \\ \hline
1.0 &	$0.75 \pm0.10 $&	$0.64 \pm0.07 $&	$0.61 \pm0.08 $&	$0.58 \pm0.08 $&	$0.55 \pm0.08 $&	$0.53 \pm0.09 $&	$0.50 \pm0.10 $&	$0.48 \pm0.10 $&	$0.46 \pm0.10 $&	$0.45 \pm0.11 $\\	
1.1 &	-&	$0.63 \pm0.07 $&	$0.60 \pm0.07 $&	$0.57 \pm0.08 $&	$0.53 \pm0.08 $&	$0.50 \pm0.08 $&	$0.47 \pm0.09 $&	$0.45 \pm0.09 $&	$0.43 \pm0.10 $&	$0.42 \pm0.10 $\\	
1.2 &	-&	-&	$0.59 \pm0.07 $&	$0.55 \pm0.07 $&	$0.51 \pm0.07 $&	$0.48 \pm0.08 $&	$0.45 \pm0.08 $&	$0.42 \pm0.08 $&	$0.40 \pm0.09 $&	$0.38 \pm0.10 $\\	
1.3 &	-&	-&	-&	$0.54 \pm0.07 $&	$0.50 \pm0.07 $&	$0.46 \pm0.07 $&	$0.42 \pm0.07 $&	$0.39 \pm0.08 $&	$0.37 \pm0.09 $&	$0.35 \pm0.10 $\\	
1.4 &	-&	-&	-&	-&	$0.48 \pm0.07 $&	$0.44 \pm0.07 $&	$0.40 \pm0.07 $&	$0.36 \pm0.07 $&	$0.34 \pm0.09 $&	$0.32 \pm0.11 $\\	
1.5 &	-&	-&	-&	-&	-&	$0.42 \pm0.06 $&	$0.37 \pm0.07 $&	$0.33 \pm0.07 $&	$0.31 \pm0.09 $&	$0.28 \pm0.11 $\\	
1.6 &	-&	-&	-&	-&	-&	-&	$0.35 \pm0.08 $&	$0.30 \pm0.08 $&	$0.27 \pm0.10 $&	$0.24 \pm0.13 $\\	
1.7 &	-&	-&	-&	-&	-&	-&	-&	$0.27 \pm0.09 $&	$0.23 \pm0.11 $&	$0.19 \pm0.07 $\\	
1.8 &	-&	-&	-&	-&	-&	-&	-&	-&	$0.18 \pm0.07 $&	$0.15 \pm0.07 $\\	
1.9 &	-&	-&	-&	-&	-&	-&	-&	-&	-&	$0.11 \pm0.08 $\\	
			\hline
	\end{tabular}}
	\caption{Average computational convergence rate (mean$\pm$STD) for $q$P$p$NWAWS on NYSE(N).}
	\label{tab:converatenysen}
\end{table*}

\begin{table*}[!htb]
	\centering	
	\scalebox{0.75}{
		\begin{tabular}{cccccccccccc}
			\hline
			\diagbox{q}{p}& &1.0&1.1 & 1.2 & 1.3 & 1.4 & 1.5 & 1.6 & 1.7 & 1.8 & 1.9 \\ \hline
			\multirow{2}{*}{1.0} & CW &$2.0603$&	$1.8979 $&	$1.9468 $&	$1.9223 $&	$1.9012 $&	$2.0409 $&	$\bf{2.0932} $&	$2.0550 $&	$2.0468 $&	$1.9474 $\\
			&SR &	$0.0563$&$0.0523 $&	$0.0538 $&	$0.0532 $&	$0.0523 $&	$0.0561 $&	$\bf{0.0574} $&	$0.0563 $&	$0.0560 $&	$0.0531 $\\
			\hline
			\multirow{2}{*}{1.1} & CW &-&	$1.9107 $&	$2.0278 $&	$1.8901 $&	$1.9175 $&	$2.0327 $&	$2.0157 $&	$2.0334 $&	$2.0160 $&	$1.9007 $\\
			&SR &-&	$0.0526 $&	$0.0561 $&	$0.0520 $&	$0.0527 $&	$0.0559 $&	$0.0553 $&	$0.0557 $&	$0.0552 $&	$0.0518 $\\
			\hline
			\multirow{2}{*}{1.2} & CW &-&	-&	$1.9781 $&	$2.0071 $&	$1.9788 $&	$2.0842 $&	$2.0296 $&	$2.0181 $&	$1.9821 $&	$1.8649 $\\
			&SR &	-&-&	$0.0545 $&	$0.0554 $&	$0.0545 $&	$0.0574 $&	$0.0557 $&	$0.0553 $&	$0.0542 $&	$0.0508 $\\
			\hline
			\multirow{2}{*}{1.3} & CW &-&	-&	-&	$1.9891 $&	$1.9682 $&	$2.0375 $&	$2.0247 $&	$1.9860 $&	$1.8868 $&	$1.7689 $\\
			&SR &-&	-&	-&	$0.0548 $&	$0.0542 $&	$0.0560 $&	$0.0556 $&	$0.0544 $&	$0.0514 $&	$0.0477 $\\	
			\hline
			\multirow{2}{*}{1.4} & CW &-&	-&	-&	-&	$1.8620 $&	$1.9303 $&	$1.9373 $&	$1.9106 $&	$1.8795 $&	$1.7889 $\\
			&SR &-&-&	-&	-&	$0.0509 $&	$0.0529 $&	$0.0530 $&	$0.0521 $&	$0.0512 $&	$0.0484 $\\
			\hline
			\multirow{2}{*}{1.5} & CW &-&	-&	-&	-&	-&	$1.7826 $&	$1.7957 $&	$1.8503 $&	$1.8596 $&	$1.8244 $\\
			&SR &-&	-&	-&	-&	-&	$0.0483 $&	$0.0487 $&	$0.0503 $&	$0.0506 $&	$0.0495 $\\
			\hline
			\multirow{2}{*}{1.6} & CW &-&	-&	-&	-&	-&	-&	$1.7073 $&	$1.7543 $&	$1.7665 $&	$1.8125 $\\
			&SR &-&	-&	-&	-&	-&	-&	$0.0457 $&	$0.0473 $&	$0.0477 $&	$0.0492 $\\	
			\hline
			\multirow{2}{*}{1.7} & CW &-&	-&	-&	-&	-&	-&	-&	$1.6923 $&	$1.7183 $&	$1.7589 $\\
			&SR &-& -&	-&	-&	-&	-&	-&	$0.0452 $&	$0.0461 $&	$0.0475 $\\	
			\hline
			\multirow{2}{*}{1.8} & CW &-&	-&	-&	-&	-&	-&	-&	-&	$1.6924 $&	$1.7223 $\\
			&SR &-&	-&	-&	-&	-&	-&	-&	-&	$0.0453 $&	$0.0463 $\\
			\hline
			\multirow{2}{*}{1.9} & CW &-&	-&	-&	-&	-&	-&	-&	-&	-&	$1.7101 $\\
			&SR &-&-&	-&	-&	-&	-&	-&	-&	-&	$0.0459 $\\	
			\hline
	\end{tabular}}
	\caption{Cumulative wealth (CW) and Sharpe Ratio (SR) of $q$P$p$NWAWS on CSI300. The CW and SR for the original setting $(q,p)=(1,2)$ are $1.7750$ and $0.0479$, respectively.}
	\label{tab:cwsrcsi}
\end{table*}
	
\begin{table*}[!htb]
	\centering	
	\scalebox{0.66}{
		\begin{tabular}{cccccccccccc}
			\hline
			\diagbox{q}{p}& &1.0&1.1 & 1.2 & 1.3 & 1.4 & 1.5 & 1.6 & 1.7 & 1.8 & 1.9 \\ \hline
			\multirow{2}{*}{1.0} & CW &$2.0561E+07$&	$4.3557E+07$&	$6.3893E+07$&	$1.7701E+08$&	$1.4948E+08$&	$2.6522E+08$&	$1.8973E+08$&	$3.0391E+08$&	$3.2967E+08$&	$4.1537E+08$\\
			&SR &$0.0935$&	$0.0968 $&	$0.0980 $&	$0.1022 $&	$0.1012 $&	$0.1036 $&	$0.1020 $&	$0.1042 $&	$0.1042 $&	$0.1046 $\\
			\hline
			\multirow{2}{*}{1.1} & CW &-&	$6.5105E+07$&	$1.0056E+08$&	$1.9101E+08$&	$1.3263E+08$&	$2.5663E+08$&	$2.5658E+08$&	$3.0408E+08$&	$4.5008E+08$&	$7.1092E+08$\\
			&SR &-&	$0.0987 $&	$0.1002 $&	$0.1027 $&	$0.1007 $&	$0.1034 $&	$0.1033 $&	$0.1039 $&	$0.1051 $&	$0.1068 $\\
			\hline
			\multirow{2}{*}{1.2} & CW &-&	-&	$1.4552E+08$&	$1.1925E+08$&	$1.6110E+08$&	$1.5340E+08$&	$1.9023E+08$&	$3.1231E+08$&	$6.0239E+08$&	$8.0655E+08$\\
			&SR & -&-&	$0.1020 $&	$0.1006 $&	$0.1015 $&	$0.1009 $&	$0.1017 $&	$0.1039 $&	$0.1060 $&	$0.1072 $\\
			\hline
			\multirow{2}{*}{1.3} & CW &-&	-&	-&	$1.1036E+08$&	$1.0286E+08$&	$1.4315E+08$&	$1.7280E+08$&	$2.4992E+08$&	$6.5705E+08$&	$\bf{8.5677E+08}$\\
			&SR &-&	-&	-&	$0.0998 $&	$0.0994 $&	$0.1005 $&	$0.1012 $&	$0.1026 $&	$0.1064 $&	$\bf{0.1075} $\\	
			\hline
			\multirow{2}{*}{1.4} & CW &-&	-&	-&	-&	$1.1330E+08$&	$1.3560E+08$&	$1.6290E+08$&	$1.7457E+08$&	$5.0738E+08$&	$7.7795E+08$\\
			&SR &-&-&	-&	-&	$0.0998 $&	$0.1001 $&	$0.1007 $&	$0.1009 $&	$0.1053 $&	$0.1071 $\\
			\hline
			\multirow{2}{*}{1.5} & CW &-&	-&	-&	-&	-&	$2.2089E+08$&	$1.2184E+08$&	$2.1543E+08$&	$4.6540E+08$&	$6.5877E+08$\\
			&SR &-& -&	-&	-&	-&	$0.1024 $&	$0.0993 $&	$0.1018 $&	$0.1049 $&	$0.1064 $\\
			\hline
			\multirow{2}{*}{1.6} & CW &-&	-&	-&	-&	-&	-&	$7.5558E+07$&	$2.3410E+08$&	$3.4107E+08$&	$5.0696E+08$\\
			&SR &-&-&	-&	-&	-&	-&	$0.0973 $&	$0.1021 $&	$0.1035 $&	$0.1052 $\\	
			\hline
			\multirow{2}{*}{1.7} & CW &-&	-&	-&	-&	-&	-&	-&	$1.2833E+08$&	$3.0709E+08$&	$5.0657E+08$\\
			&SR &-&	-&	-&	-&	-&	-&	-&	$0.0993 $&	$0.1030 $&	$0.1051 $\\	
			\hline
			\multirow{2}{*}{1.8} & CW &-&	-&	-&	-&	-&	-&	-&	-&	$2.8561E+08$&	$4.0895E+08$\\
			&SR &-&-&	-&	-&	-&	-&	-&	-&	$0.1026 $&	$0.1041 $\\
			\hline
			\multirow{2}{*}{1.9} & CW &-&	-&	-&	-&	-&	-&	-&	-&	-&	$3.2349E+08$\\
			&SR &-&-&	-&	-&	-&	-&	-&	-&	-&	$0.1031 $\\	
			\hline
	\end{tabular}}
	\caption{Cumulative wealth (CW) and Sharpe Ratio (SR) of $q$P$p$NWAWS on NYSE(N). The CW and SR for the original setting $(q,p)=(1,2)$ are $3.3183e+08$ and $0.1034$, respectively.}
	\label{tab:cwsrnysen}
\end{table*}

\subsection{Convergence Theorem}
To analyze the convergence of $q$P$p$NWAWS, we first provide several lemmas on some properties of the operators involved. The first lemma indicates that the operator $\mathbf{T}_{p,q}$ designed for the nonsingular iterates can be extended to the singular iterates.

\begin{lemma}
\label{lem:contintheorem}
Let the operator $\mathbf{T}_{p,q}$ be defined in \eqref{eqn:lqwaeta}. Then for $1\leqs q\leqs p$, $ 1\leqs p<2$,
\begin{equation}
\setlength{\abovedisplayskip}{1pt}
\setlength{\belowdisplayskip}{1pt}
\label{eqn:continlem}
\lim_{\by\rightarrow \bx,\by\notin\mathcal{S}_p}\mathbf{T}_{p,q}(\by)=\bx, \quad\forall \bx\in\mathcal{S}_p.
\end{equation}
\end{lemma}	
The proof is provided in Supplementary \ref{sm_lem:contintheorem}. Therefore, we can define
\begin{equation}
\setlength{\abovedisplayskip}{1pt}
\setlength{\belowdisplayskip}{1pt}
\label{eqn:t1extend}
\mathbf{T}_{p,q}(\bx):=\bx, \quad\forall \bx\in\mathcal{S}_p, \ 1\leqs q\leqs p,\ 1\leqs p<2.
\end{equation}
The second lemma confirms that the sequence $\{\by_{(k)}\}_{k\in\bbN}$ generated by $q$P$p$NWAWS is bounded, which is based on the descent property of $\{C_{p,q}(\by_{(k)})\}_{k\in\bbN}$.

\begin{lemma}\label{lem_bounded}
The sequence $\{\by_{(k)}\}_{k\in\bbN}$ generated by $q$P$p$NWAWS is bounded.
\end{lemma}	

The proof is provided in Supplementary \ref{sm_lem_bounded}. As emphasized above, the singular set $\mathcal{S}_p$ contains infinite points in the case $1\leqs p<2$, hence the operator $\mathbf{T}_{p,q}$ may not escape from $\mathcal{S}_p$. This is a major unavoidable obstacle for convergence analysis. To overcome this obstacle, we combine Lemma \ref{lem_bounded} and the Bolzano-Weierstrasz theorem, then there exists at least one limit point $y_*$ and a subsequence $\{\by_{(k_v)}\}_{v\in\bbN}$ such that $\lim\limits_{v\to\infty}\by_{(k_v)}=\by_*$. The following lemma indicates that there are only a finite number of limit points of $\{\by_{(k)}\}_{k\in\bbN}$ in the strictly convex case $1< p<2$, even though $\mathcal{S}_p$ has infinite points. This is a novel and significant theoretical result that plays a crucial part in the convergence proof.
\begin{lemma}\label{lem:limitfini}
If $1< p<2$, there are only a finite number of limit points of $\{\by_{(k)}\}_{k\in\bbN}$. Moreover, all these limit points have the same cost function value.
\end{lemma}
The proof is provided in Supplementary \ref{sm_lem:limitfini}. 
\begin{theorem}[Convergence Theorem]
\label{thm:contheorempq}
Let $\{\by_{(k)}\}_{k\in\bbN}$ be the iteration sequence generated by $q$P$p$NWAWS in \eqref{eqn:lqwagen}. If $\by_{(k)}$ hits the minimum point $\bx_*$ of model \eqref{eqn:lpqmediangenreal}, the characterization of minimum (Corollary \ref{cor:charnonsing} and Theorem \ref{thm:charsing}) ensures that this could be recognized and the algorithm will be stopped. Otherwise, the cost function sequence $\{C_{p,q}(\by_{(k)})\}_{k\in\bbN}$ converges. Assume $1< p<2$ in addition. If $\{\by_{(k)}\}_{k\in\bbN}$ hits $\mathcal{S}_p$ for a finite number of times, then $\{\by_{(k)}\}_{k\in\bbN}$ also converges. On the other hand, if $\{\by_{(k)}\}_{k\in\bbN}$ has a nonsingular cluster point, then this cluster point is exactly the minimum point $\bx_*$ and the entire sequence $\{\by_{(k)}\}_{k\in\bbN}$ converges to $\bx_*$.
\end{theorem}
The proof is provided in Supplementary \ref{sm_thm:contheorempq}. We provide a practical way to verify the conditions of this theorem. When the algorithm reaches the convergence tolerance, we can check whether the last few iterates are singular points. If not, we can consider that $y_{(k)}$ converges to exactly the minimum point $\bx_*$.

To end this section, we indicate that a general constant-step subgradient descent method cannot guarantee convergence in the objective value for $q$P$p$NWLP due to the varying $\lambda_*$ in Theorem \ref{thm:singmin}. Besides, $\nabla C_{p,q}(\by)$ in \eqref{eqn:costgrad} may not be Lipschitz continuous even for $\by\notin\mathcal{S}_p$. We raise a counterexample: let $C_{p,q}(\by):=\|\by\|_p^q$, $\by_1=\varepsilon\bm{1}_d> \bm{0}_d$, and $\by_2=\frac{\varepsilon}{2}\bm{1}_d$. Then 
\begin{equation}
\setlength{\abovedisplayskip}{1pt}
\setlength{\belowdisplayskip}{1pt}
\label{eqn:counter}
\frac{\|\nabla C_{p,q}(\by_1)-\nabla C_{p,q}(\by_2)\|_2}{\|\by_1-\by_2\|_2}=(2-2^{2-q})qd^{\frac{q-p}{p}}\varepsilon^{q-2}.
\end{equation}
Since $q{<}2$, \eqref{eqn:counter} implies  $\lim\limits_{{\varepsilon}\to 0}\frac{\|\nabla C_{p,q}(\by_1)-\nabla C_{p,q}(\by_2)\|_2}{\|\by_1-\by_2\|_2}=\infty$. Therefore, a general subgradient descent method cannot achieve the sublinear convergence rate $O(\frac{1}{k})$. However, practical experiments show that $q$P$p$NWAWS achieves a linear computational convergence rate, which is efficient to solve $q$P$p$NWLP. 

The whole procedure of $q$P$p$NWAWS is illustrated in Supplementary \ref{sec:solvealgo}. In each iteration, the algorithm first identifies whether the current iterate is a singular point and the dimensions where the singularity occurs. Next, it computes the de-singularity subgradient (or the normal gradient if there is no singularity). The de-singularity subgradient can then be used to determine whether the current iterate is a minimum point. If it is not, the algorithm performs a de-singularity subgradient descent step to reduce the cost function, proceeding to the next iteration until the convergence tolerance is met.

\section{Experiment Result}
We adopt the evaluating baseline in \cite{lai2024singularity} with tests from different aspects to assess the performance of the proposed $q$P$p$NWAWS. In the median reversion strategy \cite{RMR2} for online portfolio selection (\textbf{OPS}, \citealt{olpsjmlr,SSPO,SPOLC,egrmvgap}), an important step is to compute the median of the asset prices in a recent time window. In the context of this paper, it aims to find a $q$-th-powered $\ell_p$-norm median
\begin{equation}
\setlength{\abovedisplayskip}{1pt}
\setlength{\belowdisplayskip}{1pt}
\label{mod:po}
	\hat{\bx}\in\argmin_{\by\in\bbR^d} \sum_{i=1}^m \|\by-\bx_i\|_p^q, \ 1\leqs q\leqs p,\ 1\leqs p<2,
\end{equation}
where $\mathbf{x}_i\in \mathbb{R}^d$ represents the price vector of $d$ assets on the $i$-th trading day, and $\{\mathbf{x}_i\}_{i=1}^m$ contains the asset prices for the most recent $m$ trading days.

\renewcommand{\thefootnote}{\arabic{footnote}}
Experiments are conducted on six data sets: CSI300 \cite{lai2024singularity}, NYSE(N) \cite{CWMR}, FTSE100, NASDAQ100 \cite{pubdata2}, FF100, and FF100MEOP \cite{ASMCVAR}. Profiles of these data sets are shown in Table \ref{tab:datasetsprofiles}. These data sets cover financial markets from different regions like China, the United States, and the United Kingdom. They also cover different frequencies, including daily, weekly, and monthly. Their dimensionalities range from $23$ to $100$. CSI300 is extracted by \citet{lai2024singularity} from the CSI300 constituents\footnote{\url{http://www.csindex.com.cn}} of Shanghai Stock Exchange and Shenzhen Stock Exchange in China, while FF100 and FF100MEOP are extracted by \citet{ASMCVAR} from the Kenneth R. French's Data Library\footnote{\url{http://mba.tuck.dartmouth.edu/pages/faculty/ken.french/data_library.html}}. FF100 is built on ME and BE/ME, while FF100MEOP is built on ME and operating profitability, respectively. All these data sets cover a wide range of practical scenarios that are sufficient to test the performance of the proposed $q$P$p$NWAWS. Due to the page limit, the experimental results on CSI300 and NYSE(N) are presented in the main text, while those on other data sets are presented in Supplementary \ref{app:addexperiment} (Tables \ref{tab:ftsesin}$\sim$\ref{tab:ff100meopcwsr}). All these results consistently show the effectiveness of $q$P$p$NWAWS.

The experiments include four parts:
\begin{enumerate}
\item We verify that $q$P$p$NWAWS successfully reduces the cost function at singular iterates, which solves the singularity problem.
\item We validate the computational efficiency of $q$P$p$NWAWS by analyzing the number of iterations and the running time required for convergence.
\item We verify that the computational convergence rate of $q$P$p$NWAWS is a linear convergence rate.
\item By assessing the investing metrics in OPS, we demonstrate the advantages of $q$P$p$NWAWS with $1\leqs p<2$ and $1\leqs q\leqs p$. Hence $q$P$p$NWAWS is useful in a practical sense.
\end{enumerate}

We change $p$ and $q$ in $[1,1.9]$ with $q\leqs p$, which covers enough situations of $1\leqs p<2$ and $1\leqs q\leqs p$. The time window size $m$ is set as $5$ by following previous methods \cite{RMR2,PPT,AICTR,mtcvar}. The convergence tolerance thresholds are set as $Tol=10^{-4}$ and $Tol\_2=10^{-14}$, and the reducing factor $\rho$ in the line search is set as $0.1$. As the observation window moves from $t=1$ to $t=T-m+1$, there are a total of $(T-m+1)$ sets of data points $\{\mathbf{x}_i\}_{i=1}^m$. Therefore, we evaluate the average performance of $q$P$p$NWAWS by conducting the experiments for $(T-m+1)$ times on each data set. The experiments are carried out on a desktop workstation with an Intel Core i9-14900KF CPU, 64-GB DDR5 6000-MHz memory cards, and an Nvidia RTX 4080 graphics card with 16-GB independent memory.

\subsection{Solving the Singularity Problem}
We record the average number of iterations required for $q$P$p$NWAWS to successfully reduce the cost function at singular iterates. The starting iterate $\by_{(0)}$ is set as the singular point $\bx_1$. For each $(q,p)$ pair, we calculate the mean and the standard deviation (STD) of the number of iterations required on the $(T-m+1)$ sets of data points $\{\bx_i\}_{i=1}^m$, shown in Tables \ref{table1} and \ref{table11}. Results show that $q$P$p$NWAWS successfully reduces the cost function in only a few iterations, thereby solving the singularity problem. As $p$ and $q$ increase, the average number of iterations shows a decreasing trend, ranging from $4.92$ to $1.95$ on CSI300 and from $3.91$ to $1.45$ on NYSE(N). A smaller $\rho$ may lead to even fewer iterations required in the line search of the step size $\lambda_*$.

\subsection{Computational Cost and Convergence}
\label{sec:comcost}
We record the average number of iterations and the average running time for $q$P$p$NWAWS to achieve convergence in Tables \ref{tab:comcostcsi} and \ref{tab:comcostnyse}. Results show that $q$P$p$NWAWS achieves rapid convergence that the average running times are all smaller than $0.03$s and the numbers of iterations are no larger than $36$. As $p$ and $q$ increase, the average number of iterations also shows a decreasing trend, ranging from $35.93$ to $7.19$ on CSI300 and from $26.71$ to $6.50$ on NYSE(N). To summarize, $q$P$p$NWAWS successfully converges at a desirable speed.

\subsection{Computational Convergence Rate}
\label{sec:converateexperiment}
We use the following formula to assess the computational convergence rate of $q$P$p$NWAWS:
\begin{equation}
\setlength{\abovedisplayskip}{1pt}
\setlength{\belowdisplayskip}{1pt}
\frac{1}{Iter-2}\sum_{o=3}^{Iter}\frac{\|\by_{(o-1)}-\by_{(Iter)}\|_2}{\|\by_{(o-2)}-\by_{(Iter)}\|_2},
\end{equation}
where $Iter$ and $\by_{(Iter)}$ denote the total number of iterations and the final iterate, respectively. Tables \ref{tab:converatecsi300} and \ref{tab:converatenysen} show the mean and STD of the computational convergence rates for $q$P$p$NWAWS with different $(q,p)$ pairs. As $p$ and $q$ increase, the average computational convergence rate decreases from $0.8$ to $0.09$. Since they are all significantly smaller than $1$, $q$P$p$NWAWS achieves at least a linear computational convergence rate.

\subsection{Investing Performance}
\label{sec:investperform}
To further assess the effectiveness of $q$P$p$NWAWS in real-world applications, we employ two main investing metrics, the final cumulative wealth (CW) and the daily Sharpe Ratio (SR, \citealt{SHARPratio}), to conduct OPS experiments. The final CW indicates the final gain of an investing strategy at the end of the entire investment, while the SR is a kind of risk-adjusted return. We use $q$P$p$NWAWS to compute the $q$-th-powered $\ell_p$-norm median in \eqref{mod:po}, and then adopt the strategy in \cite{RMR2} to produce the CW and SR scores. Results with different $(q,p)$ pairs as well as the original setting $(q,p)=(1,2)$ in \cite{RMR2} are given in Tables \ref{tab:cwsrcsi} and \ref{tab:cwsrnysen}. They indicate that $q$P$p$NWAWS achieves the best results with $(q,p)=(1,1.6)$ on CSI300 and with $(q,p)=(1.3,1.9)$ on NYSE(N). Besides, several $(q,p)$ pairs perform better than the original setting $(q,p)=(1,2)$. These results indicate that $q$P$p$NWAWS for solving $q$P$p$NWLP is useful and advantageous with $1\leqs p<2$ and $1\leqs q\leqs p$.

\section{Conclusions and Future Works}
\label{sec:conclusion}
This paper proposes a $q$-th-Powered $\ell_p$-Norm Weiszfeld Algorithm without Singularity ($q$P$p$NWAWS) for the $q$-th-Powered $\ell_p$-Norm Weber Location Problem ($q$P$p$NWLP) with $1\leqs p<2$ and $1\leqs q\leqs p$, which includes all the rest unsolved situations in this problem. One main difficulty to solve this problem is that the singular points constitute a continuum set, so that any gradient-type algorithm may visit the singular set for infinite times. $q$P$p$NWAWS is able to characterize the subgradients and minimum at any singular or nonsingular point. If it is not a minimum point, $q$P$p$NWAWS can further reduce the cost function. Moreover, it guarantees convergence in the objective function value.

Experimental results on six real-world data sets show that $q$P$p$NWAWS successfully reduces the cost function in a few iterations at a singular point. It achieves convergence in a few iterations and shows a linear computational convergence rate. Moreover, it performs well in the online portfolio selection task that its final cumulative wealth and its Sharpe ratio with several $(q,p)$ pairs are higher than those with the original setting $(q,p)=(1,2)$. Thus $q$P$p$NWAWS and $q$P$p$NWLP with $1\leqs p<2$ and $1\leqs q\leqs p$ are useful and advantageous in practice. In future works, we will extend the de-singularity methodology to the multi-facility location problem.

\section*{Acknowledgments}
This work is supported in part by the National Natural Science Foundation of China under grants 62176103, 62276114, 62206110, in part by Guangdong Basic and Applied Basic Research Foundation under grant 2023B1515120064, in part by the Science and Technology Planning Project of Guangzhou under grants 2024A04J9896, 202206030007, in part by Guangdong-Macao Advanced Intelligent Computing Joint Laboratory under grant 2020B1212030003, in part by the Key Laboratory of Smart Education of Guangdong Higher Education Institutes, Jinan University under grant 2022LSYS003, in part by the Fundamental Research Funds for the Central Universities, JNU under grant 21623202, and in part by the Major Key Project of PCL (No. PCL2024A04).

\bibliography{qrref}

\appendix

\setcounter{table}{0}
\renewcommand{\thetable}{A\arabic{table}}

\setcounter{figure}{0}
\renewcommand{\thefigure}{A\arabic{figure}}

\onecolumn
\section*{Supplementary Material}	

\section{Solving Algorithm}
\label{sec:solvealgo}
To simplify the expressions, we convert the multiplicities from $\{\eta_i\}_{i=1}^m$ back to $\{\xi_i\}_{i=1}^m$.
\begin{breakablealgorithm}
	\caption{$q$-th-Powered $\ell_p$-Norm Weiszfeld Algorithm without Singularity ($q$P$p$NWAWS)}
	\label{alg:lqwaws}
	\begin{algorithmic}
		\REQUIRE Given $m$ distinct data points $\{\mathbf{x}_i\}_{i=1}^m$, the corresponding multiplicities $\{\xi_i\}_{i=1}^m$, the order of power $q$ and the parameter $p$ of the $\ell_p$ norm, the line search factor $0<\rho<1$ and the tolerance thresholds $Tol$ and $Tol\_2$. \\
		\STATE Initialize with a starting point $\mathbf{y}_{(0)}$.
		\WHILE{1}
		\STATE Initialize $Sing=0$ and $l=0$.
		\FOR{$t=1$ to $d$}
		\STATE Compute  $V_t(\by_{(k)})=\{i\in \{1,\dots,m\}\ s.t.\ y_{(k)}^{(t)}=x_i^{(t)}\}$ and $V_t'(\by_{(k)})=\{1,\dots,m\}\backslash V_t(\by_{(k)})$.
		\IF {$|V_t(\by_{(k)})|>0$}
		\STATE $Sing=Sing+1$.
		\ENDIF
		\ENDFOR
		\IF {$Sing=0$}
		\FOR{$t=1$ to $d$}
		\STATE Compute $y_{(k+1)}^{(t)}=\frac{\sum_{i=1}^m \xi_i\| \by_{(k)}{-}\bx_i  \|_p^{q{-}p}|y_{(k)}^{(t)}{-}x_i^{(t)}|^{p{-}2}x_i^{(t)}}{\sum_{i=1}^m \xi_i\| \by_{(k)}{-}\bx_i  \|_p^{q{-}p}|{y}_{(k)}^{(t)}{-}{x}_i^{(t)}|^{p{-}2}}$.
		\ENDFOR
		\IF {$\by_{(k+1)}=\by_{(k)}$}
		\STATE Set $\bx_*=\by_{(k)}$ and break.
		\ENDIF
		\ELSE
		\FOR{$t=1$ to $d$}
		\STATE Compute $\big(\nabla D_{p,q}(\by_{(k)})\big)^{(t)}=\sum\limits_{i\in V_t'(\by_{(k)})} q\xi_i\| \by_{(k)}-\bx_i  \|_p^{q-p}|{y}_{(k)}^{(t)}-{x}_i^{(t)}|^{p-2}(y_{(k)}^{(t)}-{x}_i^{(t)})$.
		\ENDFOR
		\IF{$\|\nabla D_{p,q}(\by_{(k)})\|_2=0$}
		\STATE Set $\bx_*=\by_{(k)}$ and break.
		\ELSIF {$p=q=1$}
		\STATE Initialize $RE_a=0$.
		\FOR{$t=1$ to $d$}
		\STATE Compute $a^{(t)}=\sum_{i\in V_t(\by_{(k)}) }\xi_i$.
		\IF{$|\big(\nabla D_{p,q}(\by_{(k)})\big)^{(t)}|>a^{(t)}$}
		\STATE $RE_a=RE_a+1$.
		\ENDIF
		\ENDFOR
		\IF{$RE_a=0$}
		\STATE Set $\bx_*=\by_{(k)}$ and break.
		\ENDIF
 		\ELSIF {$q=1$, $1<p<2$}
		\FOR {$i=1$ to $m$}
		\IF{$\by_{(k)}=\bx_i$}
		\STATE $l=i$. Break.
		\ENDIF
		\ENDFOR
		\IF {$l\neq0$ and $\|\nabla D_{p,q}(\by_{(k)})\|_r\leqs \xi_l$}
		\STATE Set $\bx_*=\bx_l$ and break.
		\ENDIF
		\ENDIF
		\STATE Compute $\sD_{p,q}(\by_{(k)})$
		by \eqref{eqn:descdirect}.		
		\STATE Set $w=0$, $\lambda_0=\|\sD_{p,q}(\by_{(k)})\|_p$.
		\WHILE {$C_{p,q}(\by_{(k)}-\lambda_w\sD_{p,q}(\by_{(k)})\geqs C_{p,q}(\by_{(k)})$}
		\STATE $\lambda_{w+1}=\rho\lambda_w$. $w\leftarrow w+1$.
		\ENDWHILE
		\STATE $\by_{(k+1)}=\by_{(k)}-\lambda_w\sD_{p,q}(\by_{(k)})$.  
		\ENDIF
		\IF {$\|\by_{(k+1)}-\by_{(k)}\|_2/\|\by_{(k)}\|_2\leqs Tol$ or $\|C_{p,q}(\by_{(k+1)})-C_{p,q}(\by_{(k)})\|_2/\|C_{p,q}(\by_{(k)})\|_2\leqs Tol\_2$}
		\STATE Set $\bx_*=\by_{(k+1)}$ and break.
		\ENDIF
		\STATE $k\leftarrow k+1$
		\ENDWHILE
		\ENSURE The output $\bx_*$.
	\end{algorithmic}	
\end{breakablealgorithm}

\section{Proofs}
\subsection{Proof of Theorem \ref{thm:nonincreasing}}\label{sm_thm:nonincreasing}
We need the following lemma from \cite{beckenbach2012inequalities} to prove this Theorem.
\begin{lemma}
\label{lem:nonincreasing2}
If $u<1$, $v<1$ and $\frac{1}{u}+\frac{1}{v}=1$, then for $a>0$, $b>0$, $a^{\frac{1}{u}}b^{\frac{1}{v}}\geqslant\frac{a}{u}+\frac{b}{v}$.
\end{lemma}
\begin{proof}[Proof of Theorem \ref{thm:nonincreasing}]
Let $\tilde{C}_{p,q}^{(t)}(a):=\sum_{i=1}^m \eta_i^q\| \by_{(k)}-\bx_i  \|_p^{q-p}|{y}_{(k)}^{(t)}-{x}_i^{(t)}|^{p-2}(a-{x}_i^{(t)})^2$ for $a\in\bbR$. Then $\tilde{C}_{p,q}^{(t)}(a)$ is strictly convex. We first analyze the case that $1<p<2$.  Since $\tilde{C}_{p,q}^{(t)'}(a)=2\sum_{i=1}^m \eta_i^q\| \by_{(k)}-\bx_i  \|_p^{q-p}|{y}_{(k)}^{(t)}-{x}_i^{(t)}|^{p-2}(a-{x}_i^{(t)})$, then $\tilde{C}_{p,q}^{(t)}(a)$ has a unique minimum at ${y}_{(k+1)}^{(t)}$, which indicates that if ${y}_{(k+1)}^{(t)}\neq{y}_{(k)}^{(t)}$, $$
\tilde{C}_{p,q}^{(t)}({y}_{(k+1)}^{(t)})<\tilde{C}_{p,q}^{(t)}({y}_{(k)}^{(t)})=\sum_{i=1}^m \eta_i^q\| \by_{(k)}-\bx_i  \|_p^{q-p}|{y}_{(k)}^{(t)}-{x}_i^{(t)}|^{p}.
$$ 
Therefore,
\begin{equation}\label{neq_tildecc}
\sum_{t=1}^d\tilde{C}_{p,q}^{(t)}({y}_{(k+1)}^{(t)})
<\sum_{t=1}^d\tilde{C}_{p,q}^{(t)}({y}_{(k)}^{(t)})=\sum_{i=1}^m\sum_{t=1}^d \eta_i^q\| \by_{(k)}-\bx_i  \|_p^{q-p}|{y}_{(k)}^{(t)}-{x}_i^{(t)}|^{p}=\sum_{i=1}^m\eta_i^q\| \by_{(k)}-\bx_i  \|_p^q.
\end{equation}
On the other hand, it follows from Lemma \ref{lem:nonincreasing2} that
\begin{align}
&\sum_{t=1}^d\tilde{C}_{p,q}^{(t)}({y}_{(k+1)}^{(t)})\nonumber\\
=&\sum_{t=1}^d\sum_{i=1}^m \eta_i^q\| \by_{(k)}-\bx_i  \|_p^{q-p}|{y}_{(k)}^{(t)}-{x}_i^{(t)}|^{p-2}|y_{(k+1)}^{(t)}-{x}_i^{(t)}|^2\nonumber\\
=&\sum_{t=1}^d\sum_{i=1}^m\eta_i^q\| \by_{(k)}-\bx_i  \|_p^{q-p}\big(|{y}_{(k)}^{(t)}-{x}_i^{(t)}|^p\big)^{\frac{{p-2}}{p}}\big(|y_{(k+1)}^{(t)}-{x}_i^{(t)}|^p\big)^{\frac{2}{p}}\nonumber\\
\geqslant&\sum_{t=1}^d\sum_{i=1}^m\eta_i^q\| \by_{(k)}-\bx_i  \|_p^{q-p}\big(\frac{p-2}{p}|{y}_{(k)}^{(t)}-{x}_i^{(t)}|^p+\frac{2}{p}|y_{(k+1)}^{(t)}-{x}_i^{(t)}|^p\big)\nonumber\\
=&(1-\frac{2}{p})\sum_{i=1}^m\eta_i^q\| \by_{(k)}-\bx_i  \|_p^q+\frac{2}{p}\sum_{i=1}^m\eta_i^q(\| \by_{(k)}-\bx_i  \|_p^q)^{\frac{q-p}{q}}(\| \by_{(k+1)}-\bx_i  \|_p^q)^{\frac{p}{q}}\nonumber\\
\geqslant&(1-\frac{2}{p})\sum_{i=1}^m\eta_i^q\| \by_{(k)}-\bx_i  \|_p^q+\frac{2}{p}\sum_{i=1}^m\eta_i^q\big(\frac{q-p}{q}\| \by_{(k)}-\bx_i  \|_p^q+\frac{p}{q}\| \by_{(k+1)}-\bx_i  \|_p^q\big)\nonumber\\
=&(1-\frac{2p}{pq})C_{p,q}(\by_{(k)})+\frac{2p}{pq}C_{p,q}(\by_{(k+1)})\nonumber\\
=&(1-\frac{2}{q})C_{p,q}(\by_{(k)})+\frac{2}{q}C_{p,q}(\by_{(k+1)}).
\label{neq_yp_1}
\end{align}
Combining both directions of the inequalities of \eqref{neq_tildecc} and \eqref{neq_yp_1}, we know that 
\begin{equation*}
(1-\frac{2}{q})C_{p,q}(\by_{(k)})+\frac{2}{q}C_{p,q}(\by_{(k+1)})<C_{p,q}(\by_{(k)}).
\end{equation*} 
Then $C_{p,q}(\by_{(k+1)})\leqs C_{p,q}(\by_{(k)})$ with equality holds only when $\by_{(k+1)}=\mathbf{T}_1( \by_{(k)})= \by_{(k)}$.	
\end{proof}

\subsection{Proof of Corollary \ref{cor:charnonsing}}\label{sm_{cor:charnonsing}}
\begin{proof}
Since $\by_{(k)}\notin \mathcal{S}_p$, we can compute that
\begin{equation}
\label{eqn:cqpt}
\big(\nabla C_{p,q}(\by_{(k)})\big)^{(t)}=\sum_{i=1}^m q\eta_i^q\| \by_{(k)}-\bx_i  \|_p^{q-p}|{y}_{(k)}^{(t)}-{x}_i^{(t)}|^{p-2}(y_{(k)}^{(t)}-{x}_i^{(t)}).
\end{equation}
Combining \eqref{eqn:lqwaeta} and \eqref{eqn:cqpt}, we obtain the following equivalence for all $t$:
\begin{equation}
\label{eqn:lqwa4}
\big(\mathbf{T}_{p,q}\big(\by_{(k)}\big)\big)^{(t)}= {y}_{(k)}^{(t)}\Longleftrightarrow\frac{1}{q}\big(\nabla C_{p,q}(\by_{(k)})\big)^{(t)}=\sum_{i=1}^m \eta_i^q\| \by_{(k)}-\bx_i  \|_p^{q-p}|{y}_{(k)}^{(t)}-{x}_i^{(t)}|^{p-2}(y_{(k)}^{(t)}-{x}_i^{(t)})=0,
\end{equation}
which indicates $\mathbf{T}_{p,q}(\by_{(k)})=\by_{(k)}\Leftrightarrow\nabla C_{p,q}(\by_{(k)})=\mathbf{0}_d$. Since $C_{p,q}(\by)$ is convex, then $\nabla C_{p,q}(\by_{(k)})=\mathbf{0}_d \Leftrightarrow$ $\by_{(k)}$ is a minimum point $\bx_*$ satisfying \eqref{eqn:lpqmediangenreal}. 
\end{proof}

\subsection{Proof of Theorem \ref{thm:charsing}}\label{sm_thm:charsing}

\begin{proof}
Let $\by+\lambda \bz$ $(\lambda>0, \bz\in\bbR^d \ \text{and}\ \|\bz\|_2=1)$ be a point displaced from $\by$ towards an arbitrary direction $\bz$. Recall the fact that
\begin{equation}\label{eq_directionsubg}
\bv\in\partial C_{p,q}(\by)
\Longleftrightarrow  \bv^{\top}\bz\leqs \frac{\ud C_{p,q}(\by+\lambda \bz)}{\ud\lambda}\arrowvert_{\lambda=0} \text{ for all }\bz.			
\end{equation}
Then we calculate the subgradient(s) $\partial C_{p,q}(\by)$ by calculating $\frac{\ud C_{p,q}(\by+\lambda \bz)}{\ud\lambda}$ first, which is given by
\begin{align}
\label{eqn:lqgradlbd}
&\frac{\ud C_{p,q}(\by+\lambda \bz)}{\ud\lambda}
{=}\sum_{i=1}^m\sum_{t=1}^d q\eta_i^q z^{(t)}\| \by+\lambda\bz-\bx_i  \|_p^{q-p}|{y}^{(t)}+\lambda z^{(t)}-{x}_i^{(t)}|^{p-2}(y^{(t)}+\lambda z^{(t)}-{x}_i^{(t)})\nonumber\\
{=}&M_{p,q}(\by,\lambda)+G_{p,q}(\by,\lambda),
\end{align}
where	
\begin{align*}
M_{p,q}(\by,\lambda)&:=\sum_{i=1}^m\sum_{t\notin U_i(\by)} q\eta_i^qz^{(t)}\| \by+\lambda\bz-\bx_i  \|_p^{q-p}|y^{(t)}+\lambda z^{(t)}-{x}_i^{(t)}|^{p-2}(y^{(t)}+\lambda z^{(t)}-{x}_i^{(t)}),\\
G_{p,q}(\by,\lambda)&:=\sum_{i=1}^m\sum_{t\in U_i(\by)}q\eta_i^q\lambda^{p-1}|z^{(t)}|^p\|\by+\lambda\bz-\bx_i\|_p^{q-p}.
\end{align*}
We compute the subgradient(s) of $C_{p,q}(\by)$ at $\by\in\mathcal{S}_p$ by considering three cases based on the values of $p$ and $q$: (a) $q=p=1$; (b) $q=1$, $1<p<2$; (c) $1<q\leqs p$, $1<p<2$.
	
Case (a): We compute the limit of $\frac{\ud}{\ud\lambda}C_{1,1}(\by+\lambda \bz)$ as $\lambda \rightarrow 0$, which is given by
\begin{equation}
\label{eqn:lpqgradlbd011}
\frac{\ud C_{1,1}(\by+\lambda \bz)}{\ud\lambda}\arrowvert_{\lambda=0}=M_{1,1}(\by,0)+\sum_{t=1}^d|z^{(t)}|\sum_{i\in V_t(\by) }\eta_i.
\end{equation}
From Definition \ref{defn:desinggrad}, $M_{1,1}(\by,0)=\nabla D_{1,1}(\by)^\top\bz$. Then \eqref{eqn:lpqgradlbd011} can be formulated as		
\begin{equation}
\label{eqn:lpqgradlbd01var}
\frac{\ud C_{1,1}(\by+\lambda \bz)}{\ud\lambda}\arrowvert_{\lambda=0}=\nabla D_{1,1}(\by)^\top\bz+\ba^{\top}|\bz|,
\end{equation}
where $\ba\in\bbR^d$ is defined by $a^{(t)}=\sum_{i\in V_t(\by) }\eta_i$ for all $i$. 	
We can easily know from \eqref{eq_directionsubg} and \eqref{eqn:lpqgradlbd01var} that the subgradients of $C_{1,1}(\by)$ at $\by\in\mathcal{S}_p$ can be formulated as
\[
\partial C_{1,1}(\by)=\{\nabla D_{1,1}(\by)+\bu\},
\]
where $\bu\in\bbR^d$ is an arbitrary vector satisfying $ -a^{(t)}\leqs u^{(t)}\leqs a^{(t)}$ for all $t$.
	
Case (b): In this case, if $\by=\bx_l$ for some $l\in\{1,\dots,m\}$, then the limit of $\frac{\ud C_{p,1}(\bx_l+\lambda \bz)}{\ud\lambda}$ as $\lambda\to0$ is given by
\begin{equation}
\label{eqn:lqgradlbd01}
\frac{\ud C_{p,1}(\bx_l+\lambda \bz)}{\ud\lambda}\arrowvert_{\lambda=0}
=M_{p,1}(\bx_l,0)+\eta_l\|\bz\|_p
=\nabla D_{p,1}(\bx_l)^{\top}\bz+\eta_l\|\bz\|_p.
\end{equation}
Let $\bu=\nabla D_{p,1}(\bx_l)^{\top}+\eta_l\bb$, where $\bb\in\bbR^d$ satisfies $\|\bb\|_r\leqs 1$. To prove that
\begin{equation}\label{eq_pacp1}
\partial C_{p,1}(\bx_l)=\{\nabla D_{p,1}(\bx_l)^{\top}+\eta_l\bb, \ \|\bb\|_r\leqs 1\},
\end{equation} 
we first prove that $\bu\in\partial C_{p,1}(\bx_l)$. From the H$\ddot{o}$lder's inequality, for $p>1$, $r>1$ with $\frac{1}{p}+\frac{1}{r}=1$, we have
\begin{equation*}
(\bu-\nabla D_{p,1}(\bx_l))^{\top}\bz\leqs\|\bu-\nabla D_{p,1}(\bx_l)\|_r\|\bz\|_p,
\end{equation*}
Then 
\begin{equation*}
\bu^{\top}\bz\leqs\nabla D_{p,1}(\bx_l)^{\top}\bz+\|\bu-\nabla D_{p,1}(\bx_l)\|_r\|\bz\|_p\leqs\nabla D_{p,1}(\bx_l)^{\top}\bz+\eta_l\|\bz\|_p,
\end{equation*}
which implies that $\bu\in\partial C_{p,1}(\bx_l)$. We next prove that all the subgradients of $C_{p,1}(\bx_l)$ can be formulated as $C_{p,1}(\bx_l)=\nabla D_{p,1}(\bx_l)^{\top}+\eta_l\bb$ by contradiction. If there exists some $\bv\in\partial C_{p,1}(\bx_l)$ such that $\bv\neq\nabla D_{p,1}(\bx_l)^{\top}+\eta_l\bb$, then $\|\bv-\nabla D_{p,1}(\bx_l)\|_r>\eta_l$. Let $\bz=\frac{|\bv-\nabla D_{p,1}(\bx_l)|^{\frac{r}{p}}}{\||\bv-\nabla D_{p,1}(\bx_l)|^{\frac{r}{p}}\|_2}$. Then $\|\bz\|_2=1$ and $\bz^p=\frac{|\bv-\nabla D_{p,1}(\bx_l)|^r}{\||\bv-\nabla D_{p,1}(\bx_l)|^{\frac{r}{p}}\|_2}$, which satisfies the condition for equality to hold in the H$\ddot{o}$lder's inequality. Then
\begin{equation*}
|\bv-\nabla D_{p,1}(\bx_l)|^{\top}\bz=\|\bv-\nabla D_{p,1}(\bx_l)\|_r\|\bz\|_p
\end{equation*}
and therefore,
\begin{equation}\label{eq_vzeta}
|\bv-\nabla D_{p,1}(\bx_l)|^{\top}\bz>\eta_l\|\bz\|_p.
\end{equation}
Define a vector $\tilde{\bz}\in\bbR^d$ such that
\[\tilde{z}^{(t)}:=
\begin{cases}
-z^{(t)}&\text{if }\big(\bv-\nabla D_{p,1}(\bx_l)\big)^{(t)}<0,\\
z^{(t)}&\text{else}.
\end{cases}
\]
Then $\|\tilde{\bz}\|_p=\|\bz\|_p$ and $(\bv-\nabla D_{p,1}(\bx_l))^{\top}\tilde{\bz}=|\bv-\nabla D_{p,1}(\bx_l)|^{\top}\bz$. It follows from \eqref{eq_vzeta} that
\begin{equation*}
(\bv-\nabla D_{p,1}(\bx_l))^{\top}\tilde{\bz}>\eta_l\|\tilde{\bz}\|_p,
\end{equation*}
which implies $
\bv^{\top}\tilde{\bz}>\nabla D_{p,1}(\bx_l)^{\top}\tilde{\bz}+\eta_l\|\tilde{\bz}\|_p$.
This contradicts that $\bv\in\partial C_{p,1}(\bx_l)$.
Then we can conclude that \eqref{eq_pacp1} holds.

If $\by\in\mathcal{S}_p\backslash\{\bx_i\}_{i=1}^m$, then the limit of $\frac{\ud}{\ud\lambda}C_{p,1}(\by+\lambda \bz)$ as $\lambda \rightarrow 0$ is given by 
\begin{equation}
\label{eqn:lpqgradlbd01p}
\frac{\ud C_{p,1}(\by+\lambda \bz)}{\ud\lambda}\arrowvert_{\lambda=0}=M_{p,1}(\bx_l,0)=\nabla D_{p,1}(\by)^\top\bz,
\end{equation}
which together with  \eqref{eqn:lpqgradlbd01p} indicates that the subgradient of $C_{p,1}(\by)$ at $\by\in\mathcal{S}_p\backslash\{\bx_i\}_{i=1}^m$ can be formulated as
\begin{equation}\label{eqn:lpqgradlbd01qvar}
\partial C_{p,1}(\by)=\{\nabla D_{p,1}(\by)\}.
\end{equation}

Case (c): We can compute that the limit of $\frac{\ud}{\ud\lambda}C_{p,q}(\by+\lambda \bz)$ as $\lambda \rightarrow 0$ is given by
\begin{equation}
\label{eqn:lpqgradlbd0pq}
\frac{\ud C_{p,q}(\by+\lambda \bz)}{\ud\lambda}\arrowvert_{\lambda=0}=M_{p,q}(\by,0)=\nabla D_{p,q}(\by)^\top\bz.
\end{equation}
From \eqref{eq_directionsubg} and \eqref{eqn:lpqgradlbd0pq}, we can easily know that the subgradient of $C_{p,q}$ at $\by\in\mathcal{S}_p$ can be formulated as
\[
\partial C_{p,q}(\by)=\{\nabla D_{p,q}(\by)\}.
\]

In conclusion, the subgradient(s) of $C_{p,q}$ at $\by\in\mathcal{S}_p$ can be represented by:
\begin{equation}\label{partialdpq}
\setlength{\abovedisplayskip}{1pt}
\setlength{\belowdisplayskip}{1pt}
\partial C_{p,q}(\by){=}
\begin{cases}
\{\nabla D_{1,1}(\by){+}\bu\}\ \text{where}\ -a^{(t)}\leqs u^{(t)}\leqs a^{(t)}, \ \ \forall t, \ &\text{if }p=q=1,\\
\{\nabla D_{p,1}(\by)\},\ &\text{if }q{=}1,1{<}p{<}2, \by\in \mathcal{S}_p\backslash\{\bx_i\}_{i=1}^m,\\
\{\nabla D_{p,1}(\bx_l)+\eta_l \bb\}\quad \text{where}\quad \|\bb\|_r\leqs 1, \ &\text{if }q{=}1,1{<}p{<}2, \by=\bx_l,\\
\{\nabla D_{p,q}(\by)\},\ &\text{if }1<q\leqs p,1<p<2,
\end{cases}
\end{equation}
where $a^{(t)}=\sum_{i\in V_t(\by) }\eta_i$ and $\|\cdot\|_r$ is the conjugate norm of $\|\cdot\|_p$ such that $\frac{1}{r}+\frac{1}{p}=1$.
From Fermat's rule, we know from \eqref{partialdpq} that a singular point $\by\in\mathcal{S}_p$ is a minimum point satisfying \eqref{eqn:lpqmediangenreal} if and only if 
\begin{equation}\label{optcondition}
\begin{cases}
|(\nabla D_{1,1}(\by))^{(t)}|\leqs a^{(t)}\ \text{for all }t, &\text{if }p=q=1,\\
\nabla D_{p,1}(\by)=\bm{0}_d,\ &\text{if }q{=}1,1{<}p{<}2\text{ and } \by\in \mathcal{S}_p\backslash\{\bx_i\}_{i=1}^m,\\
\|\nabla D_{p,1}(\bx_l)\|_r\leqs \eta_l, \ &\text{if }q{=}1,1{<}p{<}2\text{ and }\by=\bx_l,\\
\nabla D_{p,q}(\by)=\bm{0}_d&\text{if }1<q\leqs p,1<p<2.
\end{cases}
\end{equation}
\end{proof}

\subsection{Proof of Theorem \ref{thm:singmin}}\label{sm_thm:singmin}

\begin{proof}
We prove this theorem by considering three cases based on the values of $p$ and $q$: (a) $q=p=1$; (b) $q=1$, $1<p<2$; (c) $1<q\leqs p$, $1<p<2$. It suffices to show that the directional derivative along the direction $-\sD_{p,q}(\by)$ is negative.

Case (a): By setting $\mathcal{D}_{1,1}(\by)=-\frac{\sD_{1,1}(\by)}{\|\nabla D_{1,1}(\by)\|_2}=-\frac{\nabla D_{1,1}(\by)}{\|\nabla D_{1,1}(\by)\|_2}$ in \eqref{eqn:lpqgradlbd01var}, we can deduce that 
\begin{equation}
\label{direct11}
\frac{\ud C_{1,1}\big(\by+\lambda \mathcal{D}_{1,1}(\by)\big)}{\ud\lambda}\arrowvert_{\lambda=0}=-\|\nabla D_{1,1}(\by)\|_2+\mathcal{D}_{1,1}(\by)^{\top}\ba.
\end{equation}
If $\by$ is not a minimum point, then it follows from \eqref{optcondition} that $|\big(\nabla D_{1,1}(\by)\big)^{(t)}|> a^{(t)}$ for all $t$, which together with \eqref{direct11} implies that $\frac{\ud C_{1,1}(\by+\lambda \mathcal{D}_{1,1}(\by))}{\ud\lambda}\arrowvert_{\lambda=0}<0$. Then $\frac{\ud C_{1,1}(\by-\lambda \sD_{1,1}(\by))}{\ud\lambda}\arrowvert_{\lambda=0}<0$.
	
Case (b): If $\by\in \mathcal{S}_p\backslash\{\bx_i\}_{i=1}^m$, then by setting $\mathcal{D}_{p,1}(\by)=-\frac{\sD_{p,1}(\by)}{\|\nabla D_{p,1}(\by)\|_2}=-\frac{\nabla D_{p,1}(\by)}{\|\nabla D_{p,1}(\by)\|_2}$ in \eqref{eqn:lpqgradlbd01p}, we can deduce that 
\begin{equation}
\label{direct1q1}
\frac{\ud C_{p,1}(\by+\lambda \mathcal{D}_{p,1}(\by))}{\ud\lambda}\arrowvert_{\lambda=0}=-\|\nabla D_{p,1}(\by)\|_2.
\end{equation}
If $\by$ is not a minimum point, then from \eqref{optcondition}, $\|\nabla D_{p,1}(\by)\|_2>0$, which together with \eqref{direct1q1} implies that $\frac{\ud C_{p,1}(\by+\lambda \mathcal{D}_{p,1}(\by))}{\ud\lambda}\arrowvert_{\lambda=0}<0$. Then $\frac{\ud C_{p,1}(\by-\lambda \sD_{p,1}(\by))}{\ud\lambda}\arrowvert_{\lambda=0}<0$.

If $\by=\bx_l$, we set $\mathcal{D}_{p,1}(\by)=-\frac{\sD_{p,1}(\bx_l)}{\|\nabla D_{p,1}(\bx_l)\|_2}=-\frac{\big(\nabla D_{p,1}(\bx_l)\big)^{\frac{r}{p}}}{\|\nabla D_{p,1}(\bx_l)\|_2}$, where $(\bw)^{\frac{r}{p}}:=\sign(\bw) \odot |\bw|^{\frac{r}{p}}$
and $\odot$ denotes the element-wise multiplication. Then it follows from \eqref{eqn:lqgradlbd01} that
\begin{equation}\label{direct1q}
\frac{\ud C_{p,1}(\bx_l+\lambda \mathcal{D}_{p,1}(\by))}{\ud\lambda}\arrowvert_{\lambda=0}=-\frac{\nabla D_{p,1}(\bx_l)^{\top}\big(\nabla D_{p,1}(\bx_l)\big)^{\frac{r}{p}}}{\|\nabla D_{p,1}(\bx_l)\|_2}+\eta_l\frac{\|\big(\nabla D_{p,1}(\bx_l)\big)^{\frac{r}{p}}\|_p}{\|\nabla D_{p,1}(\bx_l)\|_2}.
\end{equation}
Note that $\big(|\nabla D_{p,1}(\bx_l)|^{\frac{r}{p}}\big)^p=|\nabla D_{p,1}(\bx_l)|^r$. Then from the condition for the equality to hold in the H$\ddot{o}$lder's inequality, if $p>1$, $r>1$ and $\frac{1}{p}+\frac{1}{r}=1$, then 
\begin{equation}\label{holder}
\nabla D_{p,1}(\bx_l)^{\top}\big(\nabla D_{p,1}(\bx_l)\big)^{\frac{r}{p}}=|D_{p,1}(\bx_l)|^{\top}|\nabla D_{p,1}(\bx_l)|^{\frac{r}{p}}=\|\big(\nabla D_{p,1}(\bx_l)\big)^{\frac{r}{p}}\|_p\|\nabla D_{p,1}(\bx_l)\|_r.
\end{equation}
If $\bx_l$ is not a minimum point, then from \eqref{optcondition},  $\|\nabla D_{p,1}(\bx_l)\|_r>\eta_l$, which together with \eqref{direct1q} and \eqref{holder} indicates that $
\frac{\ud C_{p,1}(\bx_l+\lambda \mathcal{D}_{p,1}(\by))}{\ud\lambda}\arrowvert_{\lambda=0}<0
$. Then $\frac{\ud C_{p,1}(\bx_l-\lambda \sD_{p,1}(\by))}{\ud\lambda}\arrowvert_{\lambda=0}<0$.

Case (c): By setting $\mathcal{D}_{p,q}(\by)=-\frac{ \sD_{p,q}(\by)}{\|\nabla D_{p,q}(\by)\|_2}=-\frac{\nabla D_{p,q}(\by)}{\|\nabla D_{p,q}(\by)\|_2}$ in \eqref{eqn:lpqgradlbd0pq}, we can deduce that 
\begin{equation}
\label{directpq}
\frac{\ud C_{p,q}(\by+\lambda \mathcal{D}_{p,q}(\by))}{\ud\lambda}\arrowvert_{\lambda=0}=-\|\nabla D_{p,q}(\by)\|_2.
\end{equation}
If $\by$ is not a minimum point, then from \eqref{optcondition}, $\|\nabla D_{p,q}(\by)\|_2>0$, which together with \eqref{directpq} implies that $\frac{\ud C_{p,q}(\by+\lambda \mathcal{D}_{p,q}(\by))}{\ud\lambda}\arrowvert_{\lambda=0}<0$. Then $\frac{\ud C_{p,q}(\by-\lambda \sD_{p,q}(\by))}{\ud\lambda}\arrowvert_{\lambda=0}<0$ 

Summarizing Cases (a)(b)(c), the direction $-\sD_{p,q}(\by)$ defined in \eqref{eqn:descdirect} is a descent direction. Hence there exists some $\lambda_*>0$ such that for any $0<\lambda\leqs\lambda_*$, $C_{p,q}(\by-\lambda \sD_{p,q}(\by))<C_{p,q}(\by)$.
\end{proof}

\subsection{Proof of Lemma \ref{lem:contintheorem}}\label{sm_lem:contintheorem}
\begin{proof}
	First, subtracting the left side of \eqref{eqn:continlem} by its right side leads to
	\begin{align}
	\label{eqn:contintheoremproof}
	&\big(\mathbf{T}_{p,q}(\by)\big)^{(t)}-x^{(t)} \nonumber\\
	=&\frac{\sum_{i=1}^m\eta_i^q\| \by-\bx_i  \|_p^{q-p}|{y}^{(t)}-{x}_i^{(t)}|^{p-2}({x}_i^{(t)}-x^{(t)})}{\sum_{i=1}^m \eta_i^q\| \by-\bx_i  \|_p^{q-p}|{y}^{(t)}-{x}_i^{(t)}|^{p-2}}\\
	=&\frac{\sum_{i\notin V_t(\bx)}\eta_i^q\| \by-\bx_i  \|_p^{q-p}|{y}^{(t)}-{x}_i^{(t)}|^{p-2}({x}_i^{(t)}-x^{(t)})}{\sum_{i\in V_t(\bx)} \eta_i^q\| \by-\bx_i  \|_p^{q-p}|{y}^{(t)}-{x}_i^{(t)}|^{p-2}+\sum_{i\notin V_t(\bx)} \eta_i^q\| \by-\bx_i  \|_p^{q-p}|{y}^{(t)}-{x}_i^{(t)}|^{p-2}},\qquad \forall 1\leqs t\leqs d.\nonumber
	\end{align}
	Since $\eta_i>0$ for all $i$, $q\leqs p$ and $p<2$, then for all $t$, 
	\begin{align*}
&\lim_{y^{(t)}\rightarrow x^{(t)},\by\notin\mathcal{S}_p}\sum_{i\notin V_t(\bx)}\eta_i^q\| \by-\bx_i  \|_p^{q-p}|{y}^{(t)}-{x}_i^{(t)}|^{p-2}({x}_i^{(t)}-x^{(t)})\\
=&\sum_{i\notin V_t(\bx)}\eta_i^q\| \bx-\bx_i  \|_p^{q-p}|{x}^{(t)}-{x}_i^{(t)}|^{p-2}({x}_i^{(t)}-x^{(t)})\\
=&c_{t,1},\\
&\lim_{y^{(t)}\rightarrow x^{(t)},\by\notin\mathcal{S}_p}\sum_{i\notin V_t(\bx)} \eta_i^q\| \by-\bx_i  \|_p^{q-p}|{y}^{(t)}-{x}_i^{(t)}|^{p-2}\\
=&\sum_{i\notin V_t(\bx)} \eta_i^q\| \bx-\bx_i  \|_p^{q-p}|{x}^{(t)}-{x}_i^{(t)}|^{p-2}\\
=&c_{t,2},\\
	&\lim_{y^{(t)}\rightarrow x^{(t)},\by\notin\mathcal{S}_p}\sum_{i\in V_t(\bx)} \eta_i^q\| \by-\bx_i  \|_p^{q-p}|{y}^{(t)}-{x}_i^{(t)}|^{p-2}\\
	=&+\infty,
	\end{align*}
for some fixed constants $c_{t,1},c_{t,2}\in \bbR$.  Therefore, 
	\begin{align}
	\label{eqn:contintheoremproof3}
	&\lim_{\by\rightarrow \bx, \by\notin\mathcal{S}_p}\|\mathbf{T}_{p,q}(\by){-}\bx\|_2 \nonumber\\
	{=}&\lim_{\by\rightarrow \bx, \by\notin\mathcal{S}_p}\big(\sum_{t=1}^d\big(\frac{ \sum_{i\notin V_t(\bx)}\eta_i^q\| \by-\bx_i  \|_p^{q-p}|{y}^{(t)}-{x}_i^{(t)}|^{p-2}({x}_i^{(t)}-x^{(t)})}{\sum_{i\in V_t(\bx)} \eta_i^q\| \by-\bx_i  \|_p^{q-p}|{y}^{(t)}-{x}_i^{(t)}|^{p-2}+ \sum_{i\notin V_t(\bx)} \eta_i^q\| \by-\bx_i  \|_p^{q-p}|{y}^{(t)}-{x}_i^{(t)}|^{p-2}}\big)^2\big)^{\frac{1}{2}}\nonumber\\
	{=}&0,
	\end{align}
	which indicates that $\mathbf{T}_{p,q}(\by)\rightarrow \bx$ as $\by\rightarrow \bx$ and $\by\notin\mathcal{S}_p$.  
\end{proof}	

\subsection{Proof of Lemma \ref{lem_bounded}}\label{sm_lem_bounded}

\begin{proof}
	It can be easily found that $\lim\limits_{\|\by\|_2\to\infty}C_{p,q}(\by)=\infty$, then we prove this lemma by yielding a contradiction. Suppose the sequence $\{\by_{(k)}\}_{k\in\bbN}$ generated by $q$P$p$NWAWS is unbounded, then there exists a subsequence $\{\by_{(k_v)}\}_{v\in\bbN}$ such that $\|\by_{(k_v)}\|_2\to\infty$ and $C_{p,q}(\by_{(k_v)})\to \infty$. However, Theorems \ref{thm:nonincreasing} and \ref{thm:singmin} indicate that $C_{p,q}(\by_{(k_v)})\leqs C_{p,q}(\by_{(0)}), \forall v$. This yields a contradiction, thus $\{\by_{(k)}\}_{k\in\bbN}$ is bounded. 
\end{proof}

\subsection{Proof of Lemma \ref{lem:limitfini}}\label{sm_lem:limitfini}

\begin{proof}
We know from Lemma \ref{lem_bounded} that $\{\by_{(k)}\}_{k\in\bbN}$ is bounded, then from the Bolzano-Weierstrasz theorem, there exists at least one point $y_*$ and a subsequence $\{\by_{(k_v)}\}_{v\in\bbN}$ such that $\lim\limits_{v\to\infty}\by_{(k_v)}=\by_*$. We then investigate the following two cases regarding whether $\by_*\notin\mathcal{S}_p$ or $\by_*\in\mathcal{S}_p$.

\noindent \textbf{Case 1.} If $\by_*\notin\mathcal{S}_p$, then we can assume that $\{\by_{(k_v)}\}_{v\in\bbN}\notin\mathcal{S}_p$. Since the operator $\mathbf{T}_{p,q}$ is continuous, then
\begin{align}
\label{eqn:contingenupdate}
\lim_{v\rightarrow \infty}\mathbf{T}_{p,q}(\by_{(k_v)})=\mathbf{T}_{p,q}(\by_*).
\end{align}
From Theorems \ref{thm:nonincreasing} and \ref{thm:singmin}, the sequence $\{C_{p,q}(\by_{(k)})\}_{k\in\bbN}$ is non-increasing. Furthermore, since $C_{p,q}(\by)\geqs0$, the sequence is bounded below. Therefore, we can conclude that the sequence $\{C_{p,q}(\by_{(k)})\}_{k\in\bbN}$ converges. According to the convergence of  $\{C_{p,q}(\by_{(k)})\}_{k\in\bbN}$, it has a limit and any subsequence of $\{C_{p,q}(\by_{(k)})\}_{k\in\bbN}$ should have the same limit. In particular, $C_{p,q}(\by_{(k_v)})$ and $C_{p,q}(\mathbf{T}_{p,q}(\by_{(k_v)}))$ are two subsequences of $C_{p,q}(\by_{(k)})$. Therefore,
\begin{align}
\label{eqn:subsequencelim}
\lim_{v\rightarrow \infty}C_{p,q}(\mathbf{T}_{p,q}(\by_{(k_v)}))=\lim_{v\rightarrow \infty}C_{p,q}(\by_{(k_v)}).\quad
\end{align}
Since $C_{p,q}$ is continuous, then it follows from (\ref{eqn:contingenupdate}) and (\ref{eqn:subsequencelim}) that
\begin{align}
\label{eqn:mincharstar}
C_{p,q}(\mathbf{T}_{p,q}(\by_*))=C_{p,q}(\by_*).
\end{align}
Since $\by_*\notin \mathcal{S}_p$, then Theorem \ref{thm:nonincreasing} and \eqref{eqn:mincharstar} indicate $\by_*=\mathbf{T}_{p,q}(\by_*)$. By Corollary \ref{cor:charnonsing}, $\by_*=\mathbf{M}$. 

\noindent \textbf{Case 2.} If $\by_*\in\mathcal{S}_p$, then we denote the limit point set of $\{\by_{(k)}\}_{k\in\bbN}$ by $\{\by_{j*}\}_{j\in \mJ}$ with $\mJ$ denoting some index set. Since $\{C_{p,q}(\by_{(k)})\}_{k\in\bbN}$ converges, we have 
\begin{equation}\label{eq_cpq}
C_{p,q}(\by_{j*})=a,\qquad \forall j\in \mJ,
\end{equation}	
for some $a\geqs 0$. Summarizing Case 1 and Case 2, $\by_*\in\{\mathbf{M}\}\bigcup\big(\mathcal{S}_p \bigcap\{\by: C_{p,q}(\by)=a\}\big)$. In other words, the limit point set of $\{\by_{(k)}\}_{k\in\bbN}$ is a subset of $\{\mathbf{M}\}\bigcup\big(\mathcal{S}_p \bigcap\{\by: C_{p,q}(\by)=a\}\big)$.

Now we only need to prove that $\mathcal{C}_{a}:=\{\by\in \bbR^d : C_{p,q}(\by)=a\}$ is a finite set, which can be done by yielding a contradiction. Suppose there exist $(n+1)$ points $\bu_1,\bu_2,\cdots,\bu_{n+1}\in \mathcal{C}_{a}$ such that $(\bu_1-\bu_{n+1})$,$(\bu_2-\bu_{n+1})$,$\cdots$,$(\bu_n-\bu_{n+1})$ are linear dependent vectors. Then there exist $\zeta_1'$,$\zeta_2'$,$\cdots$,$\zeta_n'$ that are not all zero such that:
\begin{align}
\label{eqn:lindep}
\zeta_1'(\bu_1-\bu_{n+1})+\zeta_2'(\bu_2-\bu_{n+1})+\cdots+\zeta_n'(\bu_n-\bu_{n+1})=0.
\end{align}
We then investigate the following two cases regarding whether $\sum_{o=1}^n \zeta_o'\ne 0$ or $\sum_{o=1}^n \zeta_o'=0$.

\noindent \textbf{Case 1.} If $\sum_{o=1}^n \zeta_o'\ne 0$, then we let $\zeta_j=\frac{\zeta_j'}{\sum_{o=1}^n \zeta_o'}$, $\forall j\in\{1,\cdots,n\}$. Then $\sum_{j=1}^n\zeta_j=1$ and \eqref{eqn:lindep} becomes $\bu_{n+1}=\zeta_1\bu_1+\zeta_2\bu_2+\cdots+\zeta_n\bu_n$. It follows from the strict convexity of $C_{p,q}$ and Jensen's inequality that
\begin{equation}
\zeta_1C_{p,q}(\bu_1)+\zeta_2C_{p,q}(\bu_2)+\cdots+\zeta_n C_{p,q}(\bu_n)>C_{p,q}(\bu_{n+1}),
\end{equation}
which contradicts the fact that $C_{p,q}(\bu_1)=C_{p,q}(\bu_2)=\cdots=C_{p,q}(\bu_{n+1})=a$. Since there are at most $d$ linear independent vectors in $\bbR^d$, there are at most $(d+1)$ points in $\mathcal{C}_{a}$, or else the above contradiction will occur. 

\noindent \textbf{Case 2.} If $\sum_{o=1}^n \zeta_o'=0$, then we assume $\zeta_n' \ne 0$ without loss of generality, because $\zeta_1'$,$\zeta_2'$,$\cdots$,$\zeta_n'$ are not all zero. Then \eqref{eqn:lindep} becomes
\begin{align}
\label{eqn:lindep2}
\zeta_1'\bu_1+\zeta_2'\bu_2+\cdots+\zeta_{n-1}'\bu_{n-1}=-\zeta_n'\bu_n\quad \text{and} \quad \sum_{o=1}^{n-1} \zeta_o'=-\zeta_n'.
\end{align}
Let $\zeta_j=\frac{\zeta_j'}{\sum_{o=1}^{n-1} \zeta_o'}$, $\forall j\in\{1,\cdots,n-1\}$. Then $\sum_{j=1}^{n-1}\zeta_j=1$ and \eqref{eqn:lindep2} becomes $\bu_{n}=\zeta_1\bu_1+\zeta_2\bu_2+\cdots+\zeta_{n-1}\bu_{n-1}$. Again by the strict convexity of $C_{p,q}$ and Jensen's inequality,
\begin{equation}
\zeta_1C_{p,q}(\bu_1)+\zeta_2C_{p,q}(\bu_2)+\cdots+\zeta_{n-1} C_{p,q}(\bu_{n-1})>C_{p,q}(\bu_{n}),
\end{equation}
which contradicts the fact that $C_{p,q}(\bu_1)=C_{p,q}(\bu_2)=\cdots=C_{p,q}(\bu_{n})=a$. Since there are at most $d$ linear independent vectors in $\bbR^d$, there are at most $(d+1)$ points in $\mathcal{C}_{a}$, or else the above contradiction will occur. Summarizing Case 1 and Case 2, $\mathcal{C}_{a}$ is a finite set and $\{\mathbf{M}\}\bigcup\big(\mathcal{S}_p \bigcap\mathcal{C}_{a}\big)$ is also a finite set.
\end{proof}

\subsection{Proof of Theorem \ref{thm:contheorempq}}\label{sm_thm:contheorempq}

\begin{proof}
To verify whether $\by_{(k)}$ is a minimum point, we can first check if $\by_{(k)}$ belong to $\mathcal{S}_p$. If $\by_{(k)}\notin\mathcal{S}_p$, then Corollary \ref{cor:charnonsing} implies that $\mathbf{T}_{p,q}( \by_{(k)})= \by_{(k)} \Leftrightarrow$ $\by_{(k)}=\bx_*$. Thus, we can verify whether $\by_{(k)}$ is a fixed point of $\mathbf{T}_{p,q}$. If $\by_{(k)}\in\mathcal{S}_p$, then we can directly check whether $\mathbf{0}_d\in \partial C_{p,q}(\by_{(k)})$ by \eqref{eqn:subgrad}.

By Theorems \ref{thm:nonincreasing} and \ref{thm:singmin}, we know that the sequence $\{C_{p,q}(\by_{(k)})\}_{k\in\bbN}$ is non-increasing. Furthermore, since $C_{p,q}(\by)\geqs0$, $\{C_{p,q}(\by_{(k)})\}_{k\in\bbN}$ is bounded below. Therefore, we can conclude that the sequence $\{C_{p,q}(\by_{(k)})\}_{k\in\bbN}$ converges.

If $\{\by_{(k)}\}_{k\in\bbN}$ hits $\mathcal{S}_p$ for finite times, then we can only consider the rear of the sequence that has already passed all the singular points, still denoted by $\{\by_{(k)}\}_{k\in\bbN}$ with $\{\by_{(k)}\}_{k\in\bbN}\cap\mathcal{S}_p=\emptyset$. Now we prove that the sequence $\{\by_{(k)}\}_{k\in\bbN}$ converges by showing that there exists only one limit point for $\{\by_{(k)}\}_{k\in\bbN}$. Lemma \ref{lem:limitfini} indicates that there exist only a finite number of limit points for $\{\by_{(k)}\}_{k\in\bbN}$, denoted by $\by_{1*},\by_{2*},\dots,\by_{o*}$. Consider the subsequence $\{\by_{(k_{v})}\}_{v\in\bbN}$ such that $\lim\limits_{v\to\infty}\by_{(k_{v})}=\by_{1*}$. Let $\delta=\frac{\min\limits_{i,j\in\{1,2,\dots,o\},i\neq j}\|\by_{i*}-\by_{j*}\|_2}{2}>0$. Then we consider a $\delta$-neighborhood around $\by_{1*}$, denoted by $B(\by_{1*},\delta)$. Then there exists only one limit point in $B(\by_{1*},\delta)$. We can choose a subsequence of $\{\by_{(k_{v})}\}_{v\in\bbN}$, denoted by $\{\by_{(k_{v_n})}\}_{n\in\bbN}$, such that $\{\mathbf{T}_{p,q}(\by_{(k_{v_n})})\}_{n\in\bbN}\cap B(\by_{1*},\delta)=\emptyset$, because there are other limit points $\by_{j*}\ne \by_{1*}$. Then 
\begin{equation}\label{eq_y1}
\lim\limits_{n\to\infty}\mathbf{T}_{p,q}(\by_{(k_{v_n})})\neq \by_{1*}.
\end{equation}
If $\by_{1*}\notin\mathcal{S}_p$, then it follows from the continuity of $\mathbf{T}_{p,q}$ that $\lim\limits_{n\to\infty}\mathbf{T}_{p,q}(\by_{(k_{v_n})})=\mathbf{T}_{p,q}(\by_{1*})$, which together with \eqref{eq_y1} indicates that  $\mathbf{T}_{p,q}(\by_{1*})\neq\by_{1*}$. Then from Theorem \ref{thm:nonincreasing}, $C_{p,q}(\mathbf{T}_{p,q}(\by_{1*}))<C_{p,q}(\by_{1*})$, which violates \eqref{eq_cpq}. If $\by_{1*}\in\mathcal{S}_p$, then from Lemma \ref{lem:contintheorem}, $
\lim\limits_{n\to\infty}\mathbf{T}_{p,q}(\by_{(k_{v_n})})= \by_{1*}$, which violates \eqref{eq_y1}. Therefore, there exists only one limit point for $\{\by_{(k)}\}_{k\in\bbN}$.
\end{proof}

\section{Additional Experimental Results (Tables \ref{tab:ftsesin}$\sim$\ref{tab:ff100meopcwsr})}
\label{app:addexperiment}

\begin{table*}[h]
	\centering	
	\scalebox{0.76}{
		\begin{tabular}{c c c c c c c c c c c }
			\hline
			\diagbox{q}{p}& 1.0& 1.1 & 1.2 & 1.3 & 1.4 & 1.5 & 1.6 & 1.7 & 1.8 & 1.9 \\ \hline
			1.0 &	$3.97\pm0.29$&	$2.57\pm1.3$&	$1.35\pm0.5$&	$1.24\pm0.42$&	$1.22\pm0.42$&	$1.21\pm0.41$&	$1.21\pm0.41$&	$1.21\pm0.41$&	$1.21\pm0.4$&	$1.21\pm0.41$\\
			1.1 &	-&	$2.91\pm1.67$&	$1.53\pm0.51$&	$1.41\pm0.49$&	$1.34\pm0.47$&	$1.28\pm0.45$&	$1.24\pm0.43$&	$1.22\pm0.41$&	$1.2\pm0.4$&	$1.2\pm0.4$\\
			1.2 &	-&	-&	$1.61\pm0.75$&	$1.37\pm0.48$&	$1.31\pm0.46$&	$1.27\pm0.44$&	$1.24\pm0.43$&	$1.21\pm0.41$&	$1.19\pm0.4$&	$1.19\pm0.39$\\
			1.3 &	-&	-&	-&	$1.4\pm0.6$&	$1.3\pm0.46$&	$1.26\pm0.44$&	$1.23\pm0.42$&	$1.21\pm0.4$&	$1.19\pm0.39$&	$1.18\pm0.39$\\
			1.4 &	-&	-&	-&	-&	$1.31\pm0.5$&	$1.25\pm0.43$&	$1.22\pm0.42$&	$1.2\pm0.4$&	$1.18\pm0.39$&	$1.18\pm0.38$\\
			1.5 &	-&	-&	-&	-&	-&	$1.26\pm0.45$&	$1.21\pm0.41$&	$1.19\pm0.4$&	$1.18\pm0.39$&	$1.18\pm0.38$\\
			1.6 &	-&	-&	-&	-&	-&	-&	$1.22\pm0.43$&	$1.19\pm0.39$&	$1.18\pm0.39$&	$1.18\pm0.38$\\
			1.7 &	-&	-&	-&	-&	-&	-&	-&	$1.19\pm0.4$&	$1.18\pm0.39$&	$1.17\pm0.38$\\
			1.8 &	-&	-&	-&	-&	-&	-&	-&	-&	$1.19\pm0.4$&	$1.17\pm0.38$\\
			1.9 &	-&	-&	-&	-&	-&	-&	-&	-&	-&	$1.18\pm0.38$\\
			\hline
	\end{tabular}}
	\caption{Average number of iterates for $q$P$p$NWAWS to reduce the cost function at a singular point on FTSE100 (mean$\pm$STD)}
	\label{tab:ftsesin}
\end{table*}

\begin{table*}[t]
	\centering	
	\scalebox{0.76}{
		\begin{tabular}{c c c c c c c c c c c }
			\hline
			\diagbox{q}{p}& 1.0& 1.1 & 1.2 & 1.3 & 1.4 & 1.5 & 1.6 & 1.7 & 1.8 & 1.9 \\ \hline
			1.0 &	$3.76\pm0.44$&	$2.63\pm1.56$&	$1.17\pm0.38$&	$1.08\pm0.27$&	$1.08\pm0.27$&	$1.07\pm0.26$&	$1.07\pm0.25$&	$1.07\pm0.25$&	$1.06\pm0.24$&	$1.06\pm0.24$\\
			1.1 &	-&	$3.18\pm1.8$&	$1.34\pm0.49$&	$1.21\pm0.41$&	$1.13\pm0.33$&	$1.11\pm0.31$&	$1.09\pm0.29$&	$1.08\pm0.27$&	$1.07\pm0.25$&	$1.06\pm0.23$\\
			1.2 &	-&	-&	$1.39\pm0.66$&	$1.18\pm0.39$&	$1.12\pm0.33$&	$1.1\pm0.3$&	$1.09\pm0.28$&	$1.07\pm0.26$&	$1.06\pm0.23$&	$1.06\pm0.23$\\
			1.3 &	-&	-&	-&	$1.17\pm0.39$&	$1.12\pm0.32$&	$1.09\pm0.29$&	$1.08\pm0.28$&	$1.07\pm0.25$&	$1.05\pm0.23$&	$1.05\pm0.22$\\
			1.4 &	-&	-&	-&	-&	$1.11\pm0.31$&	$1.09\pm0.29$&	$1.07\pm0.26$&	$1.06\pm0.23$&	$1.05\pm0.22$&	$1.05\pm0.21$\\
			1.5 &	-&	-&	-&	-&	-&	$1.09\pm0.28$&	$1.07\pm0.26$&	$1.06\pm0.23$&	$1.05\pm0.22$&	$1.04\pm0.21$\\
			1.6 &	-&	-&	-&	-&	-&	-&	$1.08\pm0.27$&	$1.05\pm0.23$&	$1.05\pm0.21$&	$1.04\pm0.21$\\
			1.7 &	-&	-&	-&	-&	-&	-&	-&	$1.05\pm0.22$&	$1.05\pm0.21$&	$1.04\pm0.21$\\
			1.8 &	-&	-&	-&	-&	-&	-&	-&	-&	$1.05\pm0.21$&	$1.04\pm0.21$\\
			1.9 &	-&	-&	-&	-&	-&	-&	-&	-&	-&	$1.05\pm0.21$\\
			\hline
		\end{tabular}}
	\caption{Average number of iterates for $q$P$p$NWAWS to reduce the cost function at a singular point on NASDAQ100 (mean$\pm$STD).}
	\label{tab:nasdaqsin}
\end{table*}	

\begin{table*}[t]
	\centering	
	\scalebox{0.76}{
		\begin{tabular}{c c c c c c c c c c c }
			\hline
			\diagbox{q}{p}& 1.0& 1.1 & 1.2 & 1.3 & 1.4 & 1.5 & 1.6 & 1.7 & 1.8 & 1.9 \\ \hline
			1.0 &	$2.51\pm1.19$&	$1.65\pm1.15$&	$1.2\pm0.59$&	$1.19\pm0.58$&	$1.19\pm0.56$&	$1.14\pm0.45$&	$1.11\pm0.42$&	$1.11\pm0.41$&	$1.11\pm0.4$&	$1.1\pm0.36$\\
			1.1 &	-&	$1.92\pm1.39$&	$1.2\pm0.61$&	$1.19\pm0.58$&	$1.19\pm0.56$&	$1.17\pm0.45$&	$1.12\pm0.4$&	$1.1\pm0.38$&	$1.1\pm0.37$&	$1.09\pm0.33$\\
			1.2 &	-&	-&	$1.21\pm0.61$&	$1.18\pm0.57$&	$1.19\pm0.55$&	$1.17\pm0.45$&	$1.14\pm0.42$&	$1.1\pm0.39$&	$1.09\pm0.35$&	$1.08\pm0.32$\\
			1.3 &	-&	-&	-&	$1.18\pm0.56$&	$1.19\pm0.55$&	$1.17\pm0.45$&	$1.16\pm0.44$&	$1.1\pm0.38$&	$1.09\pm0.35$&	$1.08\pm0.29$\\
			1.4 &	-&	-&	-&	-&	$1.18\pm0.53$&	$1.17\pm0.45$&	$1.17\pm0.45$&	$1.11\pm0.39$&	$1.09\pm0.34$&	$1.07\pm0.28$\\
			1.5 &	-&	-&	-&	-&	-&	$1.16\pm0.44$&	$1.16\pm0.44$&	$1.14\pm0.43$&	$1.1\pm0.35$&	$1.07\pm0.26$\\
			1.6 &	-&	-&	-&	-&	-&	-&	$1.16\pm0.44$&	$1.15\pm0.43$&	$1.14\pm0.39$&	$1.1\pm0.29$\\
			1.7 &	-&	-&	-&	-&	-&	-&	-&	$1.16\pm0.44$&	$1.16\pm0.41$&	$1.13\pm0.34$\\
			1.8 &	-&	-&	-&	-&	-&	-&	-&	-&	$1.15\pm0.41$&	$1.13\pm0.34$\\
			1.9 &	-&	-&	-&	-&	-&	-&	-&	-&	-&	$1.13\pm0.34$\\
			\hline
	\end{tabular}}
	\caption{Average number of iterates for $q$P$p$NWAWS to reduce the cost function at a singular point on FF100 (mean$\pm$STD).}
	\label{tab:ff100sin}
\end{table*}	

\begin{table*}[t]
	\centering	
	\scalebox{0.76}{
		\begin{tabular}{c c c c c c c c c c c }
			\hline
			\diagbox{q}{p}& 1.0& 1.1 & 1.2 & 1.3 & 1.4 & 1.5 & 1.6 & 1.7 & 1.8 & 1.9 \\ \hline
			1.0 &	$2.69\pm1.1$&	$1.59\pm1.15$&	$1.09\pm0.31$&	$1.08\pm0.28$&	$1.07\pm0.27$&	$1.07\pm0.26$&	$1.07\pm0.27$&	$1.07\pm0.26$&	$1.07\pm0.26$&	$1.07\pm0.26$\\
			1.1 &	-&	$1.79\pm1.24$&	$1.08\pm0.28$&	$1.08\pm0.27$&	$1.08\pm0.27$&	$1.08\pm0.27$&	$1.07\pm0.26$&	$1.07\pm0.26$&	$1.07\pm0.25$&	$1.07\pm0.25$\\
			1.2 &	-&	-&	$1.12\pm0.44$&	$1.08\pm0.27$&	$1.08\pm0.27$&	$1.07\pm0.26$&	$1.07\pm0.26$&	$1.07\pm0.25$&	$1.07\pm0.25$&	$1.07\pm0.25$\\
			1.3 &	-&	-&	-&	$1.07\pm0.26$&	$1.07\pm0.26$&	$1.07\pm0.26$&	$1.07\pm0.26$&	$1.07\pm0.25$&	$1.07\pm0.25$&	$1.07\pm0.25$\\
			1.4 &	-&	-&	-&	-&	$1.07\pm0.26$&	$1.07\pm0.26$&	$1.07\pm0.26$&	$1.07\pm0.25$&	$1.07\pm0.25$&	$1.07\pm0.25$\\
			1.5 &	-&	-&	-&	-&	-&	$1.07\pm0.26$&	$1.07\pm0.26$&	$1.07\pm0.25$&	$1.06\pm0.25$&	$1.06\pm0.25$\\
			1.6 &	-&	-&	-&	-&	-&	-&	$1.07\pm0.26$&	$1.06\pm0.25$&	$1.06\pm0.24$&	$1.06\pm0.24$\\
			1.7 &	-&	-&	-&	-&	-&	-&	-&	$1.06\pm0.24$&	$1.06\pm0.24$&	$1.06\pm0.24$\\
			1.8 &	-&	-&	-&	-&	-&	-&	-&	-&	$1.06\pm0.24$&	$1.06\pm0.24$\\
			1.9 &	-&	-&	-&	-&	-&	-&	-&	-&	-&	$1.06\pm0.24$\\
			\hline
	\end{tabular}}
	\caption{Average number of iterates for $q$P$p$NWAWS to reduce the cost function at a singular point on FF100MEOP (mean$\pm$STD).}
	\label{tab:ff100meopsin}
\end{table*}	

\begin{table*}[!htb]
	\centering	
	\scalebox{0.68}{
		\begin{tabular}{cccccccccccc}
			\hline
			\diagbox{q}{p}& &1.0 &1.1 & 1.2 & 1.3 & 1.4 & 1.5 & 1.6 & 1.7 & 1.8 & 1.9 \\ \hline
			\multirow{2}{*}{1.0} & $Time$&	$0.0074$&	$0.0598$&	$0.0625$&	$0.0617$&	$0.0582$&	$0.0541$&	$0.0515$&	$0.0484$&	$0.0464$&	$0.0446$\\	
			&$Iter$&	$42.23\pm19.91$&	$16.07\pm2.14$&	$15.17\pm2.15$&	$14.85\pm2.43$&	$14.04\pm2.5$&	$13.29\pm2.56$&	$12.65\pm2.59$&	$12.14\pm2.6$&	$11.73\pm2.63$&	$11.41\pm2.61$\\	
			\hline \multirow{2}{*}{1.1} & $Time$&	-&	$0.0611$&	$0.0621$&	$0.0584$&	$0.0539$&	$0.0496$&	$0.0465$&	$0.0442$&	$0.0418$&	$0.0404$\\	
			&$Iter$&	-&	$16.67\pm3.27$&	$15.21\pm1.8$&	$14.36\pm1.91$&	$13.37\pm2$&	$12.49\pm2.07$&	$11.84\pm2.1$&	$11.32\pm2.13$&	$10.89\pm2.13$&	$10.55\pm2.15$\\	
			\hline \multirow{2}{*}{1.2} & $Time$&	-&	-&	$0.0597$&	$0.0562$&	$0.0507$&	$0.0465$&	$0.0431$&	$0.0411$&	$0.0386$&	$0.0373$\\	
			&$Iter$&	-&	-&	$15.04\pm2.09$&	$13.81\pm1.48$&	$12.74\pm1.62$&	$11.83\pm1.68$&	$11.18\pm1.74$&	$10.66\pm1.76$&	$10.22\pm1.77$&	$9.88\pm1.83$\\	
			\hline \multirow{2}{*}{1.3} & $Time$&	-&	-&	-&	$0.0539$&	$0.0482$&	$0.0441$&	$0.0403$&	$0.0381$&	$0.0358$&	$0.0341$\\	
			&$Iter$&	-&	-&	-&	$13.65\pm2.87$&	$12.23\pm1.24$&	$11.29\pm1.35$&	$10.61\pm1.44$&	$10.11\pm1.46$&	$9.67\pm1.49$&	$9.33\pm1.54$\\	
			\hline \multirow{2}{*}{1.4} & $Time$&	-&	-&	-&	-&	$0.046$&	$0.0414$&	$0.0377$&	$0.0356$&	$0.0334$&	$0.0315$\\	
			&$Iter$&	-&	-&	-&	-&	$11.93\pm1.16$&	$10.81\pm1.05$&	$10.09\pm1.15$&	$9.58\pm1.21$&	$9.16\pm1.26$&	$8.82\pm1.31$\\	
			\hline \multirow{2}{*}{1.5} & $Time$&	-&	-&	-&	-&	-&	$0.0398$&	$0.0353$&	$0.0332$&	$0.0309$&	$0.029$\\	
			&$Iter$&	-&	-&	-&	-&	-&	$10.49\pm0.91$&	$9.62\pm0.91$&	$9.11\pm1$&	$8.68\pm1.06$&	$8.36\pm1.11$\\	
			\hline \multirow{2}{*}{1.6} & $Time$&	-&	-&	-&	-&	-&	-&	$0.0332$&	$0.0309$&	$0.0287$&	$0.0267$\\	
			&$Iter$&	-&	-&	-&	-&	-&	-&	$9.32\pm0.79$&	$8.68\pm0.83$&	$8.25\pm0.89$&	$7.91\pm0.95$\\	
			\hline \multirow{2}{*}{1.7} & $Time$&	-&	-&	-&	-&	-&	-&	-&	$0.0285$&	$0.0266$&	$0.0246$\\	
			&$Iter$&	-&	-&	-&	-&	-&	-&	-&	$8.31\pm0.69$&	$7.85\pm0.75$&	$7.48\pm0.84$\\	
			\hline \multirow{2}{*}{1.8} & $Time$&	-&	-&	-&	-&	-&	-&	-&	-&	$0.0242$&	$0.0224$\\	
			&$Iter$&	-&	-&	-&	-&	-&	-&	-&	-&	$7.46\pm0.75$&	$7.05\pm0.64$\\	
			\hline \multirow{2}{*}{1.9} & $Time$&	-&	-&	-&	-&	-&	-&	-&	-&	-&	$0.0205$\\	
			&$Iter$&	-&	-&	-&	-&	-&	-&	-&	-&	-&	$6.75\pm0.77$\\	
			\hline
	\end{tabular}}
	\caption{Average computational time (in seconds) and average number of iterations (mean$\pm$STD) for $q$P$p$NWAWS on FTSE100.}
	\label{tab:ftsecomt}
\end{table*}

\begin{table*}[!htb]
	\centering	
	\scalebox{0.68}{
		\begin{tabular}{cccccccccccc}
			\hline
			\diagbox{q}{p}& & 1.0 &1.1 & 1.2 & 1.3 & 1.4 & 1.5 & 1.6 & 1.7 & 1.8 & 1.9 \\ \hline \multirow{2}{*}{1.0} & $Time$&	$0.0104$&	$0.0674$&	$0.0925$&	$0.0931$&	$0.0872$&	$0.0566$&	$0.0531$&	$0.0728$&	$0.0479$&	$0.0462$\\
			&$Iter$&	$57.78\pm27.93$&	$17.13\pm2.68$&	$16.15\pm2.41$&	$15.96\pm2.64$&	$15.13\pm2.75$&	$14.28\pm2.81$&	$13.57\pm2.83$&	$12.99\pm2.81$&	$12.49\pm2.77$&	$12.11\pm2.75$\\
			\hline \multirow{2}{*}{1.1} & $Time$&	-&	$0.0836$&	$0.0908$&	$0.0856$&	$0.0784$&	$0.0511$&	$0.048$&	$0.0627$&	$0.0432$&	$0.0415$\\
			&$Iter$&	-&	$17.24\pm3.36$&	$15.9\pm2.09$&	$15.06\pm2.09$&	$13.93\pm2.17$&	$13.14\pm2.25$&	$12.49\pm2.25$&	$11.97\pm2.3$&	$11.52\pm2.25$&	$11.14\pm2.26$\\
			\hline \multirow{2}{*}{1.2} & $Time$&	-&	-&	$0.0882$&	$0.0813$&	$0.0738$&	$0.0472$&	$0.0443$&	$0.0415$&	$0.0397$&	$0.0379$\\
			&$Iter$&	-&	-&	$15.76\pm3.05$&	$14.41\pm1.67$&	$13.24\pm1.71$&	$12.36\pm1.79$&	$11.7\pm1.87$&	$11.2\pm1.9$&	$10.77\pm1.91$&	$10.42\pm1.9$\\
			\hline \multirow{2}{*}{1.3} & $Time$&	-&	-&	-&	$0.0779$&	$0.0602$&	$0.044$&	$0.041$&	$0.0381$&	$0.0368$&	$0.035$\\
			&$Iter$&	-&	-&	-&	$13.99\pm1.29$&	$12.65\pm1.31$&	$11.69\pm1.39$&	$11.04\pm1.49$&	$10.5\pm1.55$&	$10.1\pm1.59$&	$9.76\pm1.62$\\
			\hline \multirow{2}{*}{1.4} & $Time$&	-&	-&	-&	-&	$0.0462$&	$0.0413$&	$0.038$&	$0.0355$&	$0.034$&	$0.0321$\\
			&$Iter$&	-&	-&	-&	-&	$12.22\pm0.94$&	$11.15\pm1.03$&	$10.43\pm1.17$&	$9.9\pm1.26$&	$9.49\pm1.32$&	$9.15\pm1.35$\\
			\hline \multirow{2}{*}{1.5} & $Time$&	-&	-&	-&	-&	-&	$0.0387$&	$0.0438$&	$0.0331$&	$0.0309$&	$0.0294$\\
			&$Iter$&	-&	-&	-&	-&	-&	$10.7\pm0.76$&	$9.9\pm0.87$&	$9.35\pm1.01$&	$8.92\pm1.08$&	$8.61\pm1.13$\\
			\hline \multirow{2}{*}{1.6} & $Time$&	-&	-&	-&	-&	-&	-&	$0.0473$&	$0.0306$&	$0.0285$&	$0.0265$\\
			&$Iter$&	-&	-&	-&	-&	-&	-&	$9.47\pm0.67$&	$8.85\pm0.76$&	$8.41\pm0.87$&	$8.08\pm0.95$\\
			\hline \multirow{2}{*}{1.7} & $Time$&	-&	-&	-&	-&	-&	-&	-&	$0.0278$&	$0.0263$&	$0.0242$\\
			&$Iter$&	-&	-&	-&	-&	-&	-&	-&	$8.4\pm0.59$&	$7.94\pm0.67$&	$7.61\pm0.78$\\
			\hline \multirow{2}{*}{1.8} & $Time$&	-&	-&	-&	-&	-&	-&	-&	-&	$0.0239$&	$0.022$\\
			&$Iter$&	-&	-&	-&	-&	-&	-&	-&	-&	$7.53\pm0.59$&	$7.11\pm0.61$\\
			\hline \multirow{2}{*}{1.9} & $Time$&	-&	-&	-&	-&	-&	-&	-&	-&	-&	$0.02$\\
			&$Iter$&	-&	-&	-&	-&	-&	-&	-&	-&	-&	$6.76\pm0.53$\\
			\hline
	\end{tabular}}
	\caption{Average computational time (in seconds) and average number of iterations (mean$\pm$STD) for $q$P$p$NWAWS on NASDAQ100.}
	\label{tab:nasdaqcomt}
\end{table*}

\begin{table*}[!htb]
	\centering	
	\scalebox{0.68}{
		\begin{tabular}{cccccccccccc}
			\hline
			\diagbox{q}{p}& & 1.0 &1.1 & 1.2 & 1.3 & 1.4 & 1.5 & 1.6 & 1.7 & 1.8 & 1.9 \\ \hline
			\multirow{2}{*}{1.0} & $Time$&	$0.015$&	$0.1253$&	$0.13$&	$0.1296$&	$0.1261$&	$0.1244$&	$0.1216$&	$0.1197$&	$0.1177$&	$0.117$\\	
			&$Iter$&	$68.9\pm32.57$&	$26.12\pm27.31$&	$25.54\pm25.66$&	$25.2\pm24.24$&	$24.75\pm23.87$&	$24.21\pm23.47$&	$23.64\pm23.02$&	$23.44\pm23.29$&	$23.29\pm23.7$&	$23.15\pm24.11$\\	
			\hline \multirow{2}{*}{1.1} & $Time$&	-&	$0.1281$&	$0.1274$&	$0.126$&	$0.1211$&	$0.1185$&	$0.1134$&	$0.1106$&	$0.1069$&	$0.1055$\\	
			&$Iter$&	-&	$26.68\pm26.77$&	$25.48\pm25.71$&	$24.79\pm24.45$&	$23.98\pm23.56$&	$23.24\pm23.28$&	$22.55\pm22.94$&	$22.03\pm22.93$&	$21.66\pm23.27$&	$21.37\pm23.6$\\	
			\hline \multirow{2}{*}{1.2} & $Time$&	-&	-&	$0.1253$&	$0.1222$&	$0.1164$&	$0.111$&	$0.1073$&	$0.103$&	$0.099$&	$0.0982$\\	
			&$Iter$&	-&	-&	$25.22\pm25.49$&	$24.33\pm24.71$&	$23.38\pm24.09$&	$22.25\pm23.31$&	$21.52\pm23.11$&	$20.99\pm23.05$&	$20.59\pm23.27$&	$20.29\pm23.51$\\	
			\hline \multirow{2}{*}{1.3} & $Time$&	-&	-&	-&	$0.1163$&	$0.1317$&	$0.1051$&	$0.1013$&	$0.0974$&	$0.0943$&	$0.0928$\\	
			&$Iter$&	-&	-&	-&	$23.73\pm24.69$&	$22.68\pm24.31$&	$21.43\pm23.4$&	$20.75\pm23.27$&	$20.21\pm23.24$&	$19.81\pm23.39$&	$19.51\pm23.58$\\	
			\hline \multirow{2}{*}{1.4} & $Time$&	-&	-&	-&	-&	$0.1047$&	$0.0992$&	$0.0946$&	$0.0921$&	$0.0884$&	$0.0863$\\	
			&$Iter$&	-&	-&	-&	-&	$21.68\pm24.14$&	$20.66\pm23.7$&	$19.97\pm23.51$&	$19.47\pm23.45$&	$19.05\pm23.53$&	$18.75\pm23.67$\\	
			\hline \multirow{2}{*}{1.5} & $Time$&	-&	-&	-&	-&	-&	$0.0928$&	$0.0898$&	$0.0868$&	$0.0837$&	$0.0813$\\	
			&$Iter$&	-&	-&	-&	-&	-&	$19.87\pm23.84$&	$19.25\pm23.75$&	$18.78\pm23.77$&	$18.35\pm23.73$&	$18.01\pm23.81$\\	
			\hline \multirow{2}{*}{1.6} & $Time$&	-&	-&	-&	-&	-&	-&	$0.0839$&	$0.0817$&	$0.0789$&	$0.0764$\\	
			&$Iter$&	-&	-&	-&	-&	-&	-&	$18.57\pm23.96$&	$18.16\pm24.03$&	$17.74\pm24.07$&	$17.43\pm24.23$\\	
			\hline \multirow{2}{*}{1.7} & $Time$&	-&	-&	-&	-&	-&	-&	-&	$0.0763$&	$0.0746$&	$0.0722$\\	
			&$Iter$&	-&	-&	-&	-&	-&	-&	-&	$17.53\pm24.24$&	$17.14\pm24.24$&	$16.76\pm24.31$\\	
			\hline \multirow{2}{*}{1.8} & $Time$&	-&	-&	-&	-&	-&	-&	-&	-&	$0.0699$&	$0.0678$\\	
			&$Iter$&	-&	-&	-&	-&	-&	-&	-&	-&	$16.61\pm24.45$&	$16.19\pm24.52$\\	
			\hline \multirow{2}{*}{1.9} & $Time$&	-&	-&	-&	-&	-&	-&	-&	-&	-&	$0.0631$\\	
			&$Iter$&	-&	-&	-&	-&	-&	-&	-&	-&	-&	$15.63\pm24.72$\\
			\hline
	\end{tabular}}
	\caption{Average computational time (in seconds) and average number of iterations (mean$\pm$STD) for $q$P$p$NWAWS on FF100.}
	\label{tab:ff100comt}
\end{table*}

\begin{table*}[!htb]
	\centering	
	\scalebox{0.68}{
		\begin{tabular}{cccccccccccc}
			\hline
			\diagbox{q}{p}& & 1.0 &1.1 & 1.2 & 1.3 & 1.4 & 1.5 & 1.6 & 1.7 & 1.8 & 1.9 \\ \hline
			\multirow{2}{*}{1.0} & $Time$&	$0.0144$&	$0.0965$&	$0.1022$&	$0.1053$&	$0.105$&	$0.1031$&	$0.0987$&	$0.0965$&	$0.0956$&	$0.0935$\\
			&$Iter$&	$65.22\pm33.00$&	$16.55\pm4.04$&	$16.5\pm3.71$&	$17\pm3.62$&	$16.88\pm3.62$&	$16.5\pm3.72$&	$16.18\pm3.77$&	$15.86\pm3.8$&	$15.56\pm3.83$&	$15.21\pm3.9$\\
			\hline \multirow{2}{*}{1.1} & $Time$&	-&	$0.0938$&	$0.0974$&	$0.0978$&	$0.0954$&	$0.0923$&	$0.088$&	$0.0848$&	$0.0829$&	$0.0799$\\
			&$Iter$&	-&	$16.26\pm2.91$&	$15.81\pm2.81$&	$15.87\pm2.8$&	$15.57\pm2.78$&	$15.11\pm2.79$&	$14.66\pm2.83$&	$14.21\pm2.86$&	$13.8\pm2.89$&	$13.41\pm2.95$\\
			\hline \multirow{2}{*}{1.2} & $Time$&	-&	-&	$0.0961$&	$0.0946$&	$0.0903$&	$0.0857$&	$0.0805$&	$0.0779$&	$0.0751$&	$0.0725$\\
			&$Iter$&	-&	-&	$15.76\pm2.93$&	$15.38\pm2.59$&	$14.82\pm2.45$&	$14.2\pm2.46$&	$13.65\pm2.46$&	$13.15\pm2.48$&	$12.77\pm2.49$&	$12.35\pm2.53$\\
			\hline \multirow{2}{*}{1.3} & $Time$&	-&	-&	-&	$0.089$&	$0.0854$&	$0.0786$&	$0.0749$&	$0.072$&	$0.0687$&	$0.0664$\\
			&$Iter$&	-&	-&	-&	$14.79\pm2.09$&	$14.04\pm2.04$&	$13.37\pm2.05$&	$12.79\pm2.08$&	$12.32\pm2.12$&	$11.9\pm2.14$&	$11.55\pm2.19$\\
			\hline \multirow{2}{*}{1.4} & $Time$&	-&	-&	-&	-&	$0.0774$&	$0.0719$&	$0.0683$&	$0.0666$&	$0.0628$&	$0.0601$\\
			&$Iter$&	-&	-&	-&	-&	$13.12\pm1.5$&	$12.47\pm1.57$&	$11.9\pm1.63$&	$11.48\pm1.7$&	$11.08\pm1.74$&	$10.72\pm1.84$\\
			\hline \multirow{2}{*}{1.5} & $Time$&	-&	-&	-&	-&	-&	$0.0647$&	$0.0617$&	$0.0596$&	$0.0568$&	$0.0539$\\
			&$Iter$&	-&	-&	-&	-&	-&	$11.57\pm1.18$&	$11.05\pm1.24$&	$10.61\pm1.34$&	$10.25\pm1.38$&	$9.93\pm1.48$\\
			\hline \multirow{2}{*}{1.6} & $Time$&	-&	-&	-&	-&	-&	-&	$0.0558$&	$0.0536$&	$0.0504$&	$0.0481$\\
			&$Iter$&	-&	-&	-&	-&	-&	-&	$10.24\pm0.96$&	$9.84\pm1.02$&	$9.47\pm1.09$&	$9.15\pm1.2$\\
			\hline \multirow{2}{*}{1.7} & $Time$&	-&	-&	-&	-&	-&	-&	-&	$0.0481$&	$0.0452$&	$0.0429$\\
			&$Iter$&	-&	-&	-&	-&	-&	-&	-&	$9.13\pm0.77$&	$8.75\pm0.83$&	$8.45\pm0.91$\\
			\hline \multirow{2}{*}{1.8} & $Time$&	-&	-&	-&	-&	-&	-&	-&	-&	$0.0402$&	$0.0379$\\
			&$Iter$&	-&	-&	-&	-&	-&	-&	-&	-&	$8.07\pm0.62$&	$7.77\pm0.71$\\
			\hline \multirow{2}{*}{1.9} & $Time$&	-&	-&	-&	-&	-&	-&	-&	-&	-&	$0.0323$\\
			&$Iter$&	-&	-&	-&	-&	-&	-&	-&	-&	-&	$7.04\pm0.52$\\
			\hline
	\end{tabular}}
	\caption{Average computational time (in seconds) and average number of iterations (mean$\pm$STD) for $q$P$p$NWAWS on FF100MEOP.}
	\label{tab:ff100meopcomt}
\end{table*}

\begin{table*}[t]
	\centering	
	\scalebox{0.8}{
		\begin{tabular}{c c c c c c c c c c c}
			\hline
			\diagbox{q}{p}& 1.0& 1.1 & 1.2 & 1.3 & 1.4 & 1.5 & 1.6 & 1.7 & 1.8 & 1.9 \\ \hline
			1.0 &	$0.84\pm0.06$&	$0.65\pm0.04$&	$0.64\pm0.04$&	$0.62\pm0.05$&	$0.59\pm0.05$&	$0.56\pm0.06$&	$0.53\pm0.07$&	$0.51\pm0.08$&	$0.49\pm0.09$&	$0.47\pm0.09$\\
			1.1 &	-&	$0.63\pm0.05$&	$0.62\pm0.04$&	$0.59\pm0.05$&	$0.55\pm0.06$&	$0.51\pm0.06$&	$0.48\pm0.07$&	$0.46\pm0.08$&	$0.44\pm0.08$&	$0.42\pm0.09$\\
			1.2 &	-&	-&	$0.61\pm0.05$&	$0.57\pm0.04$&	$0.53\pm0.05$&	$0.49\pm0.05$&	$0.46\pm0.06$&	$0.43\pm0.07$&	$0.41\pm0.07$&	$0.39\pm0.08$\\
			1.3 &	-&	-&	-&	$0.56\pm0.04$&	$0.51\pm0.04$&	$0.47\pm0.04$&	$0.43\pm0.05$&	$0.40\pm0.06$&	$0.37\pm0.07$&	$0.35\pm0.07$\\
			1.4 &	-&	-&	-&	-&	$0.50\pm0.03$&	$0.44\pm0.03$&	$0.40\pm0.04$&	$0.36\pm0.05$&	$0.33\pm0.06$&	$0.31\pm0.06$\\
			1.5 &	-&	-&	-&	-&	-&	$0.42\pm0.03$&	$0.37\pm0.03$&	$0.33\pm0.04$&	$0.3\pm0.04$&	$0.27\pm0.05$\\
			1.6 &	-&	-&	-&	-&	-&	-&	$0.35\pm0.03$&	$0.3\pm0.02$&	$0.26\pm0.03$&	$0.23\pm0.04$\\
			1.7 &	-&	-&	-&	-&	-&	-&	-&	$0.26\pm0.02$&	$0.22\pm0.02$&	$0.18\pm0.03$\\
			1.8 &	-&	-&	-&	-&	-&	-&	-&	-&	$0.18\pm0.03$&	$0.14\pm0.02$\\
			1.9 &	-&	-&	-&	-&	-&	-&	-&	-&	-&	$0.09\pm0.03$\\
			\hline
	\end{tabular}}
	\caption{Average computational convergence rate (mean$\pm$STD) for $q$P$p$NWAWS on FTSE100.}
	\label{tab:ftseconr}
\end{table*}

\begin{table*}[t]
	\centering	
	\scalebox{0.8}{
		\begin{tabular}{c c c c c c c c c c c }
			\hline
			\diagbox{q}{p}& 1.0&1.1 & 1.2 & 1.3 & 1.4 & 1.5 & 1.6 & 1.7 & 1.8 & 1.9 \\ \hline
			1.0 &	$0.88\pm0.06$&	$0.66\pm0.04$&	$0.65\pm0.04$&	$0.63\pm0.05$&	$0.61\pm0.06$&	$0.58\pm0.06$&	$0.55\pm0.07$&	$0.53\pm0.08$&	$0.51\pm0.09$&	$0.49\pm0.09$\\
			1.1 &	-&	$0.62\pm0.05$&	$0.63\pm0.05$&	$0.6\pm0.05$&	$0.56\pm0.06$&	$0.53\pm0.07$&	$0.5\pm0.07$&	$0.48\pm0.08$&	$0.46\pm0.08$&	$0.44\pm0.09$\\
			1.2 &	-&	-&	$0.62\pm0.05$&	$0.58\pm0.04$&	$0.54\pm0.05$&	$0.5\pm0.06$&	$0.47\pm0.06$&	$0.44\pm0.07$&	$0.42\pm0.08$&	$0.41\pm0.08$\\
			1.3 &	-&	-&	-&	$0.57\pm0.04$&	$0.52\pm0.04$&	$0.47\pm0.05$&	$0.44\pm0.06$&	$0.41\pm0.06$&	$0.38\pm0.07$&	$0.36\pm0.07$\\
			1.4 &	-&	-&	-&	-&	$0.5\pm0.03$&	$0.45\pm0.03$&	$0.41\pm0.04$&	$0.37\pm0.05$&	$0.34\pm0.06$&	$0.32\pm0.07$\\
			1.5 &	-&	-&	-&	-&	-&	$0.43\pm0.03$&	$0.38\pm0.03$&	$0.34\pm0.04$&	$0.3\pm0.05$&	$0.28\pm0.06$\\
			1.6 &	-&	-&	-&	-&	-&	-&	$0.35\pm0.02$&	$0.3\pm0.03$&	$0.26\pm0.04$&	$0.23\pm0.04$\\
			1.7 &	-&	-&	-&	-&	-&	-&	-&	$0.27\pm0.02$&	$0.22\pm0.02$&	$0.19\pm0.03$\\
			1.8 &	-&	-&	-&	-&	-&	-&	-&	-&	$0.18\pm0.02$&	$0.14\pm0.02$\\
			1.9 &	-&	-&	-&	-&	-&	-&	-&	-&	-&	$0.09\pm0.02$\\
			\hline
	\end{tabular}}
	\caption{Average computational convergence rate (mean$\pm$STD) for $q$P$p$NWAWS on NASDAQ100.}
	\label{tab:nasdaqconr}
\end{table*}

\begin{table*}[t]
	\centering	
	\scalebox{0.8}{
		\begin{tabular}{c c c c c c c c c c c }
			\hline
			\diagbox{q}{p}& 1.0&1.1 & 1.2 & 1.3 & 1.4 & 1.5 & 1.6 & 1.7 & 1.8 & 1.9 \\ \hline
			1.0 &	$0.89\pm0.08$&	$0.66\pm0.2$&	$0.67\pm0.19$&	$0.67\pm0.17$&	$0.66\pm0.18$&	$0.65\pm0.17$&	$0.64\pm0.16$&	$0.63\pm0.17$&	$0.62\pm0.18$&	$0.61\pm0.18$\\
			1.1 &	-&	$0.68\pm0.11$&	$0.68\pm0.11$&	$0.67\pm0.12$&	$0.65\pm0.12$&	$0.63\pm0.13$&	$0.62\pm0.13$&	$0.6\pm0.15$&	$0.58\pm0.16$&	$0.57\pm0.17$\\
			1.2 &	-&	-&	$0.67\pm0.12$&	$0.65\pm0.12$&	$0.63\pm0.13$&	$0.6\pm0.13$&	$0.58\pm0.14$&	$0.56\pm0.15$&	$0.55\pm0.17$&	$0.53\pm0.19$\\
			1.3 &	-&	-&	-&	$0.64\pm0.12$&	$0.61\pm0.14$&	$0.58\pm0.14$&	$0.55\pm0.15$&	$0.53\pm0.16$&	$0.51\pm0.18$&	$0.5\pm0.2$\\
			1.4 &	-&	-&	-&	-&	$0.58\pm0.14$&	$0.55\pm0.15$&	$0.52\pm0.16$&	$0.5\pm0.18$&	$0.48\pm0.18$&	$0.46\pm0.22$\\
			1.5 &	-&	-&	-&	-&	-&	$0.52\pm0.16$&	$0.48\pm0.17$&	$0.46\pm0.18$&	$0.43\pm0.2$&	$0.42\pm0.22$\\
			1.6 &	-&	-&	-&	-&	-&	-&	$0.44\pm0.18$&	$0.42\pm0.2$&	$0.39\pm0.21$&	$0.37\pm0.22$\\
			1.7 &	-&	-&	-&	-&	-&	-&	-&	$0.37\pm0.21$&	$0.34\pm0.23$&	$0.32\pm0.23$\\
			1.8 &	-&	-&	-&	-&	-&	-&	-&	-&	$0.3\pm0.26$&	$0.26\pm0.25$\\
			1.9 &	-&	-&	-&	-&	-&	-&	-&	-&	-&	$0.21\pm0.27$\\
			\hline
	\end{tabular}}
	\caption{Average computational convergence rate (mean$\pm$STD) for $q$P$p$NWAWS on FF100.}
	\label{tab:ff100conr}
\end{table*}

\begin{table*}[t]
	\centering	
	\scalebox{0.8}{
		\begin{tabular}{c c c c c c c c c c c }
			\hline
			\diagbox{q}{p}& 1.0&1.1 & 1.2 & 1.3 & 1.4 & 1.5 & 1.6 & 1.7 & 1.8 & 1.9 \\ \hline
			1.0 &	$0.88\pm0.09$&	$0.64\pm0.14$&	$0.64\pm0.13$&	$0.64\pm0.11$&	$0.64\pm0.1$&	$0.62\pm0.09$&	$0.61\pm0.1$&	$0.6\pm0.09$&	$0.59\pm0.09$&	$0.58\pm0.1$\\
			1.1 &	-&	$0.64\pm0.05$&	$0.63\pm0.05$&	$0.62\pm0.05$&	$0.6\pm0.05$&	$0.59\pm0.06$&	$0.57\pm0.06$&	$0.55\pm0.07$&	$0.54\pm0.08$&	$0.52\pm0.09$\\
			1.2 &	-&	-&	$0.62\pm0.05$&	$0.61\pm0.04$&	$0.58\pm0.05$&	$0.56\pm0.05$&	$0.53\pm0.06$&	$0.51\pm0.07$&	$0.5\pm0.07$&	$0.48\pm0.08$\\
			1.3 &	-&	-&	-&	$0.59\pm0.04$&	$0.56\pm0.04$&	$0.53\pm0.05$&	$0.5\pm0.06$&	$0.48\pm0.06$&	$0.46\pm0.07$&	$0.44\pm0.08$\\
			1.4 &	-&	-&	-&	-&	$0.53\pm0.03$&	$0.49\pm0.04$&	$0.46\pm0.05$&	$0.44\pm0.06$&	$0.41\pm0.06$&	$0.39\pm0.07$\\
			1.5 &	-&	-&	-&	-&	-&	$0.45\pm0.03$&	$0.42\pm0.04$&	$0.39\pm0.04$&	$0.37\pm0.05$&	$0.34\pm0.07$\\
			1.6 &	-&	-&	-&	-&	-&	-&	$0.37\pm0.02$&	$0.34\pm0.03$&	$0.31\pm0.04$&	$0.29\pm0.05$\\
			1.7 &	-&	-&	-&	-&	-&	-&	-&	$0.29\pm0.02$&	$0.26\pm0.03$&	$0.23\pm0.04$\\
			1.8 &	-&	-&	-&	-&	-&	-&	-&	-&	$0.2\pm0.02$&	$0.17\pm0.03$\\
			1.9 &	-&	-&	-&	-&	-&	-&	-&	-&	-&	$0.1\pm0.01$\\	
			\hline
	\end{tabular}}
	\caption{Average computational convergence rate (mean$\pm$STD) for $q$P$p$NWAWS on FF100MEOP.}
	\label{tab:ff100meopconr}
\end{table*}

\begin{table*}[!htb]
	\centering	
	\scalebox{0.7}{
		\begin{tabular}{cccccccccccc}
			\hline
			\diagbox{q}{p}& &1.0&1.1 & 1.2 & 1.3 & 1.4 & 1.5 & 1.6 & 1.7 & 1.8 & 1.9 \\ \hline
			 \multirow{2}{*}{1.0} & CW&	$45.4645$&	$75.5157$&	$\bf 78.5411$&	$68.5563$&	$64.3221$&	$60.4422$&	$60.0741$&	$56.2867$&	$51.9187$&	$48.8151$\\
			&SR&	$0.1061$&	$0.1166$&	$\bf0.1172$&	$0.1143$&	$0.1129$&	$0.1116$&	$0.1114$&	$0.1102$&	$0.1087$&	$0.1076$\\
			\hline \multirow{2}{*}{1.1} & CW&	-&	$70.6336$&	$73.1394$&	$62.1984$&	$61.933$&	$64.6889$&	$59.7303$&	$54.9652$&	$52.7812$&	$53.578$\\
			&SR&	-&	$0.1153$&	$0.1158$&	$0.1124$&	$0.1121$&	$0.1128$&	$0.1113$&	$0.1097$&	$0.109$&	$0.1092$\\
			\hline \multirow{2}{*}{1.2} & CW&	-&	-&	$66.0843$&	$57.5627$&	$58.7638$&	$60.6815$&	$57.1283$&	$53.0324$&	$54.1895$&	$56.1977$\\
			&SR&	-&	-&	$0.1138$&	$0.1109$&	$0.1111$&	$0.1115$&	$0.1104$&	$0.109$&	$0.1094$&	$0.1101$\\
			\hline \multirow{2}{*}{1.3} & CW&	-&	-&	-&	$50.2791$&	$54.8675$&	$55.6076$&	$53.038$&	$51.177$&	$51.75$&	$54.4712$\\
			&SR&	-&	-&	-&	$0.1084$&	$0.1098$&	$0.1099$&	$0.1091$&	$0.1084$&	$0.1086$&	$0.1095$\\
			\hline \multirow{2}{*}{1.4} & CW&	-&	-&	-&	-&	$51.8308$&	$52.7809$&	$50.6453$&	$50.9802$&	$51.4462$&	$54.142$\\
			&SR&	-&	-&	-&	-&	$0.1087$&	$0.109$&	$0.1082$&	$0.1083$&	$0.1084$&	$0.1093$\\
			\hline \multirow{2}{*}{1.5} & CW&	-&	-&	-&	-&	-&	$52.0359$&	$51.4891$&	$51.3103$&	$51.1554$&	$54.3556$\\
			&SR&	-&	-&	-&	-&	-&	$0.1087$&	$0.1085$&	$0.1084$&	$0.1083$&	$0.1094$\\
			\hline \multirow{2}{*}{1.6} & CW&	-&	-&	-&	-&	-&	-&	$52.727$&	$52.3873$&	$52.002$&	$53.9062$\\
			&SR&	-&	-&	-&	-&	-&	-&	$0.1089$&	$0.1088$&	$0.1086$&	$0.1092$\\
			\hline \multirow{2}{*}{1.7} & CW&	-&	-&	-&	-&	-&	-&	-&	$54.0681$&	$54.2055$&	$54.4498$\\
			&SR&	-&	-&	-&	-&	-&	-&	-&	$0.1094$&	$0.1093$&	$0.1093$\\
			\hline \multirow{2}{*}{1.8} & CW&	-&	-&	-&	-&	-&	-&	-&	-&	$56.2806$&	$55.3808$\\
			&SR&	-&	-&	-&	-&	-&	-&	-&	-&	$0.11$&	$0.1096$\\
			\hline \multirow{2}{*}{1.9} & CW&	-&	-&	-&	-&	-&	-&	-&	-&	-&	$55.1256$\\
			&SR&	-&	-&	-&	-&	-&	-&	-&	-&	-&	$0.1095$\\	
			\hline
	\end{tabular}}
	\caption{Cumulative wealth (CW) and Sharpe Ratio (SR) of $q$P$p$NWAWS on FTSE100. The CW and SR for the original setting $(q,p)=(1,2)$ are $50.1765$ and $0.1081$, respectively.}
	\label{tab:ftsecwsr}
\end{table*}

\begin{table*}[!htb]
	\centering	
	\scalebox{0.65}{
		\begin{tabular}{cccccccccccc}
			\hline
			\diagbox{q}{p}& &1.0&1.1 & 1.2 & 1.3 & 1.4 & 1.5 & 1.6 & 1.7 & 1.8 & 1.9 \\ \hline
			 \multirow{2}{*}{1.0} & CW&	$1.3759$&	$7.6786$&	$\bf8.1831$&	$8.0839$&	$7.0116$&	$6.1615$&	$5.5029$&	$5.1351$&	$4.8972$&	$4.6473$\\
			&SR&	$0.0393$&	$0.0873$&	$\bf0.0892$&	$0.0892$&	$0.0854$&	$0.0818$&	$0.0783$&	$0.0761$&	$0.0746$&	$0.073$\\\hline
			\multirow{2}{*}{1.1} & CW&	-&	$6.4397$&	$7.2775$&	$7.0786$&	$6.1042$&	$5.3003$&	$4.9398$&	$4.7081$&	$4.5468$&	$4.3697$\\
			&SR&	-&	$0.0823$&	$0.0858$&	$0.0852$&	$0.0813$&	$0.0772$&	$0.0751$&	$0.0735$&	$0.0724$&	$0.0712$\\\hline
			\multirow{2}{*}{1.2} & CW&	-&	-&	$6.7903$&	$6.4068$&	$5.2845$&	$4.7138$&	$4.545$&	$4.458$&	$4.2981$&	$4.0487$\\
			&SR&	-&	-&	$0.0839$&	$0.0824$&	$0.0771$&	$0.0738$&	$0.0726$&	$0.0719$&	$0.0708$&	$0.069$\\\hline
			\multirow{2}{*}{1.3} & CW&	-&	-&	-&	$6.1334$&	$4.8934$&	$4.2499$&	$4.2272$&	$4.2161$&	$4.0197$&	$3.8176$\\
			&SR&	-&	-&	-&	$0.0813$&	$0.0749$&	$0.0708$&	$0.0705$&	$0.0703$&	$0.0689$&	$0.0674$\\\hline
			\multirow{2}{*}{1.4} & CW&	-&	-&	-&	-&	$4.9208$&	$4.368$&	$4.1212$&	$3.9263$&	$3.7898$&	$3.6361$\\
			&SR&	-&	-&	-&	-&	$0.0752$&	$0.0716$&	$0.0698$&	$0.0683$&	$0.0673$&	$0.066$\\\hline
			\multirow{2}{*}{1.5} & CW&	-&	-&	-&	-&	-&	$4.3312$&	$3.8162$&	$3.5874$&	$3.4547$&	$3.3215$\\
			&SR&	-&	-&	-&	-&	-&	$0.0714$&	$0.0676$&	$0.0658$&	$0.0646$&	$0.0635$\\\hline
			\multirow{2}{*}{1.6} & CW&	-&	-&	-&	-&	-&	-&	$3.5415$&	$3.2637$&	$3.1451$&	$3.0999$\\
			&SR&	-&	-&	-&	-&	-&	-&	$0.0655$&	$0.0631$&	$0.062$&	$0.0615$\\\hline
			\multirow{2}{*}{1.7} & CW&	-&	-&	-&	-&	-&	-&	-&	$3.0056$&	$2.8606$&	$2.8844$\\
			&SR&	-&	-&	-&	-&	-&	-&	-&	$0.0607$&	$0.0593$&	$0.0595$\\\hline
			\multirow{2}{*}{1.8} & CW&	-&	-&	-&	-&	-&	-&	-&	-&	$2.6534$&	$2.6911$\\
			&SR&	-&	-&	-&	-&	-&	-&	-&	-&	$0.0572$&	$0.0576$\\\hline
			\multirow{2}{*}{1.9} & CW&	-&	-&	-&	-&	-&	-&	-&	-&	-&	$2.5943$\\
			&SR&	-&	-&	-&	-&	-&	-&	-&	-&	-&	$0.0566$\\
			\hline
	\end{tabular}}
	\caption{Cumulative wealth (CW) and Sharpe Ratio (SR) of $q$P$p$NWAWS on NASDAQ100. The CW and SR for the original setting $(q,p)=(1,2)$ are $4.4003$ and $0.0714$, respectively.}
	\label{tab:nasdaqcwsr}
\end{table*}	

\begin{table*}[!htb]
	\centering	
	\scalebox{0.65}{
		\begin{tabular}{cccccccccccc}
			\hline
			\diagbox{q}{p}& &1.0&1.1 & 1.2 & 1.3 & 1.4 & 1.5 & 1.6 & 1.7 & 1.8 & 1.9 \\ \hline
			 \multirow{2}{*}{1.0} & CW&	$\bf78.2733$&	$66.1289$&	$65.4248$&	$62.4756$&	$60.487$&	$60.581$&	$59.5244$&	$58.6327$&	$58.2801$&	$57.9964$\\
			&SR&	$\bf0.1684$&	$0.161$&	$0.1606$&	$0.1584$&	$0.1572$&	$0.1575$&	$0.1571$&	$0.1564$&	$0.1561$&	$0.1558$\\
			\hline \multirow{2}{*}{1.1} & CW&	-&	$71.5448$&	$70.1542$&	$68.908$&	$63.4955$&	$64.5905$&	$62.6362$&	$61.8175$&	$61.2322$&	$61.8582$\\
			&SR&	-&	$0.1636$&	$0.1625$&	$0.1617$&	$0.1589$&	$0.1592$&	$0.1587$&	$0.1582$&	$0.1578$&	$0.158$\\
			\hline \multirow{2}{*}{1.2} & CW&	-&	-&	$71.6484$&	$67.7225$&	$69.1062$&	$65.4196$&	$64.0382$&	$63.7063$&	$63.1599$&	$63.7408$\\
			&SR&	-&	-&	$0.163$&	$0.161$&	$0.1618$&	$0.1599$&	$0.1593$&	$0.1592$&	$0.1589$&	$0.1591$\\
			\hline \multirow{2}{*}{1.3} & CW&	-&	-&	-&	$71.903$&	$77.1214$&	$66.4467$&	$65.345$&	$65.4356$&	$64.8402$&	$65.3721$\\
			&SR&	-&	-&	-&	$0.1635$&	$0.1653$&	$0.1609$&	$0.1603$&	$0.1602$&	$0.1598$&	$0.16$\\
			\hline \multirow{2}{*}{1.4} & CW&	-&	-&	-&	-&	$68.3831$&	$69.2298$&	$66.8632$&	$67.1484$&	$66.3637$&	$66.7174$\\
			&SR&	-&	-&	-&	-&	$0.1616$&	$0.1623$&	$0.161$&	$0.161$&	$0.1605$&	$0.1607$\\
			\hline \multirow{2}{*}{1.5} & CW&	-&	-&	-&	-&	-&	$70.7036$&	$69.0821$&	$68.5513$&	$67.7569$&	$68.025$\\
			&SR&	-&	-&	-&	-&	-&	$0.1628$&	$0.162$&	$0.1617$&	$0.1612$&	$0.1613$\\
			\hline \multirow{2}{*}{1.6} & CW&	-&	-&	-&	-&	-&	-&	$71.661$&	$70.1756$&	$69.0988$&	$69.2098$\\
			&SR&	-&	-&	-&	-&	-&	-&	$0.1632$&	$0.1624$&	$0.1618$&	$0.1618$\\
			\hline \multirow{2}{*}{1.7} & CW&	-&	-&	-&	-&	-&	-&	-&	$71.4712$&	$70.1408$&	$69.7693$\\
			&SR&	-&	-&	-&	-&	-&	-&	-&	$0.163$&	$0.1623$&	$0.1621$\\
			\hline \multirow{2}{*}{1.8} & CW&	-&	-&	-&	-&	-&	-&	-&	-&	$70.967$&	$70.3666$\\
			&SR&	-&	-&	-&	-&	-&	-&	-&	-&	$0.1626$&	$0.1623$\\
			\hline \multirow{2}{*}{1.9} & CW&	-&	-&	-&	-&	-&	-&	-&	-&	-&	$71.0373$\\
			&SR&	-&	-&	-&	-&	-&	-&	-&	-&	-&	$0.1626$\\	
			\hline
	\end{tabular}}
	\caption{Cumulative wealth (CW) and Sharpe Ratio (SR) of $q$P$p$NWAWS on FF100. The CW and SR for the original setting $(q,p)=(1,2)$ are $60.2157$ and $0.1569$, respectively.}
	\label{tab:ff100cwsr}
\end{table*}

\begin{table*}[!htb]
	\centering	
	\scalebox{0.65}{
		\begin{tabular}{cccccccccccc}
			\hline
			\diagbox{q}{p}& &1.0&1.1 & 1.2 & 1.3 & 1.4 & 1.5 & 1.6 & 1.7 & 1.8 & 1.9 \\ \hline \multirow{2}{*}{1.0} & CW&	$\bf117.2168$&	$99.7265$&	$97.2394$&	$95.0089$&	$95.7351$&	$95.961$&	$96.2792$&	$95.7062$&	$95.7199$&	$95.7733$\\
			&SR&	$\bf0.1799$&	$0.1737$&	$0.1727$&	$0.1719$&	$0.1723$&	$0.1724$&	$0.1726$&	$0.1725$&	$0.1725$&	$0.1726$\\
			\hline \multirow{2}{*}{1.1} & CW&	-&	$99.9262$&	$98.5504$&	$96.8576$&	$96.7715$&	$96.9106$&	$97.2259$&	$97.5471$&	$97.8084$&	$98.3201$\\
			&SR&	-&	$0.1736$&	$0.1731$&	$0.1725$&	$0.1726$&	$0.1727$&	$0.1729$&	$0.1731$&	$0.1732$&	$0.1735$\\
			\hline \multirow{2}{*}{1.2} & CW&	-&	-&	$100.3027$&	$98.9845$&	$98.8171$&	$98.9833$&	$99.4045$&	$99.7495$&	$100.1719$&	$100.8471$\\
			&SR&	-&	-&	$0.1737$&	$0.1733$&	$0.1734$&	$0.1735$&	$0.1737$&	$0.1739$&	$0.1741$&	$0.1744$\\
			\hline \multirow{2}{*}{1.3} & CW&	-&	-&	-&	$102.0236$&	$101.7188$&	$101.856$&	$102.1336$&	$102.4828$&	$102.8048$&	$103.4895$\\
			&SR&	-&	-&	-&	$0.1744$&	$0.1744$&	$0.1745$&	$0.1747$&	$0.1749$&	$0.175$&	$0.1753$\\
			\hline \multirow{2}{*}{1.4} & CW&	-&	-&	-&	-&	$104.8142$&	$105.0038$&	$105.1909$&	$105.4119$&	$105.5615$&	$106.1576$\\
			&SR&	-&	-&	-&	-&	$0.1755$&	$0.1756$&	$0.1758$&	$0.1759$&	$0.176$&	$0.1762$\\
			\hline \multirow{2}{*}{1.5} & CW&	-&	-&	-&	-&	-&	$108.0347$&	$108.0789$&	$108.1193$&	$108.2833$&	$108.7254$\\
			&SR&	-&	-&	-&	-&	-&	$0.1766$&	$0.1767$&	$0.1768$&	$0.1769$&	$0.1771$\\
			\hline \multirow{2}{*}{1.6} & CW&	-&	-&	-&	-&	-&	-&	$110.7385$&	$110.7341$&	$110.5933$&	$111.0338$\\
			&SR&	-&	-&	-&	-&	-&	-&	$0.1776$&	$0.1777$&	$0.1777$&	$0.1779$\\
			\hline \multirow{2}{*}{1.7} & CW&	-&	-&	-&	-&	-&	-&	-&	$112.885$&	$112.6995$&	$113.0243$\\
			&SR&	-&	-&	-&	-&	-&	-&	-&	$0.1784$&	$0.1783$&	$0.1785$\\
			\hline \multirow{2}{*}{1.8} & CW&	-&	-&	-&	-&	-&	-&	-&	-&	$114.5054$&	$114.8594$\\
			&SR&	-&	-&	-&	-&	-&	-&	-&	-&	$0.1789$&	$0.1791$\\
			\hline \multirow{2}{*}{1.9} & CW&	-&	-&	-&	-&	-&	-&	-&	-&	-&	$116.612$\\
			&SR&	-&	-&	-&	-&	-&	-&	-&	-&	-&	$0.1796$\\
			\hline
	\end{tabular}}
	\caption{Cumulative wealth (CW) and Sharpe Ratio (SR) of $q$P$p$NWAWS on FF100MEOP. The CW and SR for the original setting $(q,p)=(1,2)$ are $95.7632$ and $0.1726$, respectively.}
	\label{tab:ff100meopcwsr}
\end{table*}

\end{document}